\documentclass[11 pt]{amsart}
\usepackage{amsfonts}
\usepackage{amsmath,amscd}
\usepackage{fullpage}
\usepackage{centernot} 
\usepackage{pb-diagram} 
\usepackage{mathrsfs}
\usepackage[OT2,T1]{fontenc}

\usepackage[color]{showkeys}

\definecolor{refkey}{rgb}{1,1,1}
\definecolor{labelkey}{rgb}{1,1,1}
\definecolor{cite}{rgb}{0.9451,0.2706,0.4941}
\definecolor{ruri}{rgb}{0.0078,0.4022,0.8010}

\usepackage[%
bookmarks=true,bookmarksnumbered=true,%
colorlinks=true,linkcolor=ruri,citecolor=red%
 ]{hyperref}

\makeindex 

\def\F{{\rm \mathbb{F}}}
\def\Z{{\rm \mathbb{Z}}}

\def\Q{{\rm \mathbb{Q}}}

\def\C{{\rm \mathbb{C}}}

\def\P{{\rm \mathbb{P}}}

\def\a{{\rm \mathfrak{a}}}
\def\b{{\rm \mathfrak{b}}}
\def\c{{\rm \mathfrak{c}}}

\def\O{{\rm \mathcal{O}}}
\def\A{{\rm \mathcal{A}}}

\def\l{{\rm \lambda}}

\def\Nm{{\rm Nm}}
\def\GRH{{\rm GRH}}
\def\lcm{{\rm lcm}}
\def\Aut{{\rm Aut}}

\def\Tr{{\rm Tr}}

\def\NS{{\rm NS}}
\def\Pic{{\rm Pic}}
\def\Cl{{\rm Cl}}
\def\amp{{\rm amp}}
\def\Disc{{\rm Disc}}
\def\e{{\rm \epsilon}}
\def\SL{{\rm SL}}
\def\Ell{{\rm Ell}}
\def\GL{{\rm GL}}

\def\Ext{{\rm Ext}}

\def\Div{{\rm Div}}

\def\Hom{{\rm Hom}}
\def\End{{\rm End}}

\def\sm{{\rm sm}}
\def\va{{\rm va}}
\def\prim{{\rm prim}}
\def\mprim{{\rm mp}}

\def\id{{\rm id}}

\def\mod{{\rm mod \hspace{1 mm}}}

\numberwithin{equation}{section}

\newtheorem{theorem}{Theorem}[section]
\newtheorem{lemma}[theorem]{Lemma}

\newtheorem{corollary}[theorem]{Corollary}
\newtheorem{proposition}[theorem]{Proposition}
\newenvironment{definition}[1][Definition]{\begin{trivlist}
\item[\hskip \labelsep {\bfseries #1}]}{\end{trivlist}}
\newenvironment{notation}[1][Notation]{\begin{trivlist}
\item[\hskip \labelsep {\bfseries #1}]}{\end{trivlist}}
\newenvironment{example}[1][Example]{\begin{trivlist}
\item[\hskip \labelsep {\bfseries #1}]}{\end{trivlist}}
\newenvironment{remark}[1][Remark]{\begin{trivlist}
\item[\hskip \labelsep {\bfseries #1}]}{\end{trivlist}}

\usepackage{marginnote,color,changepage}
\usepackage[usenames,dvipsnames]{xcolor}

\def\shownotes{\def\inline##1##2##3{ \begin{adjustwidth}{3mm}{7mm}\mbox{}\par \noindent
{\color{##1}\hspace{-1.9cm}{\large ##2}\vspace{-\baselineskip}\\##3}
\newline\end{adjustwidth}} \def\inlinewide##1##2##3{ \begin{adjustwidth}{0mm}{0cm}\mbox{}\par \noindent
{\color{##1}\hspace{-1.6cm}{\large ##2}\vspace{-\baselineskip}\\##3}
\newline\end{adjustwidth}}  \def\marg##1##2##3{\marginnote{\color{##1}{\large ##2}\\{\small ##3}}[-.8cm]}}

\shownotes

\begin{document}

\title{N\'eron-Severi groups of product abelian surfaces}
\author{Julian Rosen}
\author{Ariel Shnidman}
\address{Department of Mathematics, University of Michigan, 530 Church Street, Ann Arbor, MI 48109, U.S.A.}
\email{shnidman@umich.edu}
\email{j2rosen@uwaterloo.ca}

\begin{abstract}
We give a natural parameterization of the N\'eron-Severi group of a product $A = E\times E'$ of two elliptic curves in terms of quadratic forms.  As an application, we determine (in the non-CM case) whether $A$ contains a smooth curve of any fixed genus.  We also determine whether $A$ admits a very ample line bundle of any fixed degree.  In particular, we determine which of these abelian surfaces embed in $\P^4$, i.e.\ which come from the Horrocks-Mumford bundle.     
\end{abstract}


\maketitle
\tableofcontents

\section{Introduction}
Let $A$ be an abelian variety over an algebraically closed field $k$ of characteristic 0.  The N\'eron-Severi group $\NS(A) = \Pic(A)/\Pic^0(A)$ is a free abelian group of finite rank $\rho(A)$, called the Picard number of $A$.  The group $\NS(A)$ appears often in the study of abelian varieties, but perhaps its primary importance is as the natural setting to study polarizations on $A$.  A \textit{polarization} $L \in \NS(A)$ is the class of an ample line bundle.  The most important invariant of a polarization is its \emph{degree} $d = d(L)$, which can be defined in various ways: 
\begin{equation*}d = \chi(L) = h^0(A,L) = \frac{1}{(\dim A )!}(L^{\dim A}).\end{equation*}

When $A$ is an abelian surface (i.e.\ $\dim A = 2$), the degree function is the quadratic form on $\NS(A)$ whose corresponding symmetric bilinear pairing is the intersection pairing.  While $\rho(A)$ is an isogeny invariant and is totally understood in terms of the endomorphism algebra $\End(A)_\Q$, the isomorphism class of the integral quadratic space $\NS(A)$ is \textit{not} an isogeny invariant, and is hence more subtle.  In particular, it is (in general) a difficult question to determine whether $A$ has a line bundle of any fixed degree $d$.  

There are other properties of polarizations which are often of interest.  
\begin{definition}We say $L$ is \textit{smooth} if it is represented by $\O(C)$ for some smooth curve $C$, necessarily of genus $d + 1$. 
\end{definition}
\begin{definition}  We say $L$ is \textit{very ample} if it is the class of a very ample line bundle.
\end{definition}

Given a particular abelian surface $A$ and a fixed integer $d \geq 1$, it is natural to ask whether there exists in $\NS(A)$ a smooth (resp. very ample) polarization of any fixed degree $d \geq 1$.  Beyond the question of existence, one can try to \textit{count} such polarizations up to $\Aut(A)$-equivalence.  If $A$ is not simple, one can also count the classes in $\NS(A)$ represented by line bundles $\O(E)$, where $E \subset A$ is an elliptic curve.  These classes have degree 0 and are therefore not polarizations, but up to $\Aut(A)$-equivalence there are finitely many such classes.       

The goal of this paper is to introduce a new approach to these counting questions, when $A$ is a product of two elliptic curves $E \times E'$.  There are two reasons why, from our point of view, these are the most interesting abelian surfaces.  First of all, product abelian surface have large Picard number (either 2, 3, or 4), hence $\NS(A)$ is more complicated than usual.\footnote{The generic abelian surface has Picard number 1, while surfaces with real multiplication have $\rho(A) = 2$ and those with quaternionic multiplication have $\rho(A) = 3$.}  The second reason is that the presence of elliptic curves on $A$ makes detection of smooth and very ample polarizations more subtle (see Theorem \ref{reider}).  

So assume for the rest of the introduction that $A = E \times E'$, and to make matters interesting also assume that $E$ and $E'$ are isogenous. Then $\rho(A)$ is either 3 or 4, depending on whether $\End(E) = \Z$ or not; we refer to the first case as the non-CM case and the second as the CM case.  Our approach is to attach to each $L \in \NS(A)$ the quadratic form 
\[q_L:\Hom(E,A)\to\Z,\]
\[ f\mapsto\deg(f^*L).\]

It turns out that this association identifies $\NS(A)$ with a certain natural space of quadratic forms.  In the non-CM case, it is the space of integral binary quadratic forms which are degenerate over $\Z/m\Z$, where $m$ is the degree of a minimal isogeny $E \to E'$.  In the CM case, it is a space of integral binary $\O$-Hermitian forms, where $\O$ is an order in an imaginary quadratic field.      

In both cases, the association $L \mapsto q_L$ is $\Aut(A)$-equivariant and translates familiar geometric properties of line bundles into familiar properties of quadratic forms.  In the non-CM case, for example, the set of $\Aut(A)$-equivalence classes of polarizations of a fixed degree correspond to (a union of) class groups of binary quadratic forms of fixed discriminant. We set up this parameterization and the accompanying dictionary between geometry and quadratic forms in Sections \ref{param} (the non-CM case) and \ref{singular} (the CM-case).  

\subsection*{Results}
In the rest of the paper, we show the utility of this point of view by answering the aforementioned existence and counting questions for polarizations on non-CM surfaces $A$.  The answers in the CM-case require extra ideas and will be discussed in a separate paper.  The results in the non-CM case are somewhat surprising in that they depend (in a very mild way) on the generalized Riemann hypothesis (GRH).  In what follows we write $A_m$ for any non-CM abelian surface $A = E \times E'$ such that the minimal isogeny $E \to E'$ has degree $m$.  

\begin{theorem}\label{mainsmooth}
Suppose $m,d \geq 1$.  If we assume $\GRH$, then $A_m$ admits a smooth polarization of degree $d$ if and only if at least one of the following conditions holds:
\begin{enumerate}
\item $(d,m) > 1$.
\item $d$ is composite.
\item  $md$ is not on the following list of $21$ integers: 
\begin{equation}
\label{grhlist}
\tag{$\star$}
1,2,4,6,10,12,18,22,28,30,42,58,60,70,78,102,130,190,210, 330,462.
\end{equation}
\end{enumerate}
If we do not assume $\GRH$, then the same statement is true except it is possible that the list \eqref{grhlist} should include one extra integer.  In particular, there are $($unconditionally$)$ only finitely many pairs $(m,d)$ for which $A_m$ does not admit a smooth polarization of degree $d$.  
\end{theorem}

For very ample polarizations, we have a similar kind of result, which refers to the following set of exceptional numbers:

\begin{align*}
S = \{&1,2,3,4,5, 6, 7, 8, 9, 10, 12, 13, 15, 16, 18, 21, 22, 24, 25, 28, 30, 33, 37, 40, 42, 45, 48, 57, 58, 60, 70,\\ & 72, 78, 85, 88, 93, 102, 105, 112, 120, 130, 133, 165, 168, 177, 190, 210, 232, 240, 253, 273, 280, 312,\\ & 330, 345, 357, 385, 408, 462, 520, 760, 840, 1320, 1365, 1848.\}
\end{align*}

\begin{theorem}\label{mainvample}
Let $m \geq 1$ and suppose $d \geq 5$.
If $d$ is not a prime or twice a prime, then $A_m$ admits a very ample polarization of degree $d$.
Let $p$ be a prime and assume $\GRH$ is true.  Then $A_m$ admits a very ample polarization of degree $p$ if and only if at least one of the following conditions hold:
\begin{enumerate}
\item $p|m$.
\item $mp \not\in S$.
\end{enumerate}  
Similarly, $A_m$ admits a very ample polarization of degree $2p$ if and only if at least one of the following conditions hold:
\begin{enumerate}
\item $p|m$.
\item $2||m$ or $16|m$.
\item $2mp \not\in S$.  
\end{enumerate}  
If $\GRH$ is false, then the same result holds, except that the set $S$ of exceptions should contain at most two more integers.  In particular, there are $($unconditionally$)$ finitely many pairs $(m,d)$ with $d \geq 5$, such that $A_m$ does not admit a very ample polarization of degree $d$.    
\end{theorem}

\begin{remark}
The condition $d \geq 5$ is necessary for $L$ to be very ample.  Indeed, an abelian surface cannot embed in $\P^{d-1}$ for $d < 5$.  
\end{remark}

Both of these theorems have interesting and concrete consequences.  For instance, Theorem \ref{mainsmooth} allows us to determine whether $A_m$ has a smooth curve of genus $g$, for any integer $g$.  We see in particular that $A_m$ has smooth curves of every genus for all large enough $m$.  When we specialize to genus $g = 2$ (i.e polarizations of degree $d = 1$), we recover a theorem of Kani \cite[Theorem 5]{kani5}, which determines which of the product surfaces $A_m$ are Jacobians.\footnote{It is well known that $A$ is a Jacobian if and only if it contains a smooth principal polarization.}  

There is a particularly interesting consequence of Theorem \ref{mainvample}.  It is well known that every abelian surface $A$ embeds in $\P^5$, and no abelian surface can embed in $\P^3$.  Moreover, $A$ embeds in $\P^4$ if and only if $A$ is the zero-locus of a section of the Horrocks-Mumford bundle on $\P^4$ \cite[Theorem 5.2]{hm} if and only if $A$ admits a very ample polarization $L$ of degree $d = 5$.  Theorem \ref{mainvample} therefore allows us to determine which product surfaces $A_m$ embed in $\P^4$:
\begin{corollary}
Assume $\GRH$.  Then the surface $A_m$ embeds in $\P^4$ if and only if $m$ is not one of the following integers:
$$1, 2, 3, 6, 8, 9, 12, 14, 17, 21, 24, 26, 33, 38, 42, 48, 56, 66, 69, 77, 104, 152, 168, 264, 273.$$
Unconditionally, there are at most two more values of $m$ for which $A_m$ does not embed in $\P^4$.  
\end{corollary}

The key idea in the proofs of Theorems \ref{mainsmooth} and \ref{mainvample} is that non-very ample line bundles should correspond under our dictionary to quadratic forms which are 2-torsion in the class group.  Moreover, non-smooth line bundles should correspond to binary quadratic forms which are decomposable, i.e.\ those of the form $ax^2 + cy^2$.  These forms are always 2-torsion, corresponding to the fact very ample line bundle are smooth (by Bertini's theorem).  In fact, if we write $\Pic(\O_N)$ for the class group\footnote{Recall that if $N < 0$, then the class group of positive definite quadratic forms of discriminant $N$ is isomorphic to the Picard group of the (imaginary) quadratic order $\O_N$ of discriminant $N$.} of discriminant $N$, then we prove:    

\begin{theorem}\label{smcurves}
Let $m,d \geq 1$.  Then $A_m$ admits a smooth polarization of degree $d$ if and only if $md\geq 2$ and at least one of the following conditions is satisfied:
\begin{itemize}
\item $d$ is composite.
\item $(m,d) > 1$.
\item $md$ is odd or divisible by $8$.
\item $\Pic(\O_{-4md})$ is not $2$-torsion.
\end{itemize}
\end{theorem}

\begin{theorem}\label{main}
Suppose $m\geq 1$, $d \geq 5$.
\begin{itemize}
\item If $d$ is not a prime or twice a prime, then $A_m$ admits a very ample polarization of degree $d$.

\item If $d = p$ is prime, then $A_m$ admits a very ample polarization of degree $d$ if and only if either $p|m$ or $\Pic(\O_{-4md})$ is not $2$-torsion.

\item If $d = 2p$ is twice a prime, then $A_m$ admits a very ample polarization of degree $d$ if and only if either $p|m$,  $\Pic(\O_{-4md})$ is not $2$-torsion, $2||m$, or $16|m$.
\end{itemize}
\end{theorem}

Theorems \ref{mainsmooth} and \ref{mainvample} follow from Theorems \ref{smcurves} and \ref{main}, by a well known result of Weinberger \cite{wein},  which determines the numbers $N$ for which $\Pic(\O_{-4N})$ is a $2$-torsion group.
\begin{remark}
A positive integer $N$ is called \textit{idoneal} if $\Pic(\O_{-4N})$ is a 2-torsion group.  Every integer $N$ in the set $S$ above is idoneal, and Euler conjectured that $S$ is equal to the set of idoneal numbers.  Weinberger proved Euler's conjecture assuming $\GRH$, and showed unconditionally that the set of idoneal numbers contains at most two integers not in $S$.  See \cite[Cor.\ 23]{kani4} for more details.          
\end{remark}

\begin{remark}
These results concern the \textit{existence} of smooth and very ample polarizations on $A_m$, but do not address the question of {\it counting} polarizations of degree $d$ and of counting elliptic curves on $A_m$ (the case $d = 0$).  In principle, our methods give formulas for these counts for every single value of $d$ and $m$, but we have not written down the formula in the general case as it would be unwieldy.  For formulas in the cases of most interest, see Corollary \ref{polcount}, Proposition \ref{ECs}, Corollary \ref{dprime}, and Theorem \ref{vacount}.     
\end{remark}


\subsection*{The proofs}Theorem \ref{smcurves} is proven in Section 6 and Theorem \ref{main} is proven in Sec.\ 7. The main tool used in the proof of Theorems \ref{smcurves} and \ref{main} is a deep theorem of Reider (Theorem \ref{reider}) which gives an intersection-theoretic criterion to determine whether a polarization $L$ is very ample or not.  This criterion has a very simple interpretation in terms of the quadratic form $q_L$.  For example, in the case where $A = E \times E$, the criterion says that $L$ is smooth if and only if $q_L$ does not represent 1, and $L$ is very ample if and only if $q_L$ does not represent 1 or 2.  In fact, we require a strengthened version of Reider's theorem (Theorem \ref{precisereider}) which describes the locus of polarized abelian surfaces which are not very ample, sitting inside the 3-dimensional moduli space of $d$-polarized abelian surfaces.  This locus is (roughly speaking) a union of modular diagonal quotient surfaces, a fact which may be interesting in its own right.  

The other important tool is the use of Atkin-Lehner operators.  The group $W(m)$ of Atkin-Lehner involutions on $X_0(m)$ acts naturally on the set $\Aut(A)\backslash \NS(A)$. This comes from the fact that $A = E \times E'$ may have other decompositions into a product of elliptic curves.  Under our dictionary, these operators correspond to the subgroup of 2-torsion classes in an appropriate class group.  This subgroup happens to be isomorphic to $W(m)$ and acts on the class group by translation.  The key fact is that the association $L \mapsto q_L$ is equivariant for the $W(m)$ actions on both sides.  The details are discussed at the end of Section \ref{param}.

Finally, the proofs involve a case-by-case analysis, depending on the values of $d$ and $m$ (mod 8) and whether $d$ and $m$ are coprime.  Each of the cases has its own idiosyncrasies, and some of them are independently interesting.  For example, in some cases we need to construct explicit families of non-very ample polarizations, which gives rise to an explicit embedding of $X_0(m)$ inside the modular diagonal quotients mentioned earlier (see the proof Proposition \ref{divby8}).

\subsection*{Related work}
There is a long history of studying principal (i.e.\ degree 1) polarizations on split abelian surfaces in terms of quadratic forms.  See for example \cite{hay}, \cite{hay-ni}, \cite{lange}, \cite{kani5}, and \cite{kani2}.  The use of the quadratic form $q_L$ appears implicitly in \cite{hay-ni}, \cite{kani5}, and \cite[Appendix]{serre}. 

It is especially interesting to compare the methods of \cite{kani5} and \cite{kani2} with our own. Kani also uses quadratic forms to count smooth principal polarizations $L$ (so $d = 1$) on product abelian surfaces $A$. To such an $L$ he attaches the quadratic form $Q_L(D) := (D.L)^2 - 2(D.D)$ on $\NS(A)/\langle L\rangle$ which has rank $\rho(A) - 1$.  When $A = A_m$,  this quadratic space has rank 2 and discriminant $-16m$, as opposed to our quadratic space which has rank 2 and discriminant $-4m$.  The connection between these two quadratic forms is as follows (to simplify things, assume $m \not \equiv 3$ (mod 4)).  If $e: \Pic(\O_{-16m}) \to \Pic(\O_{-4m})$ is the natural surjective map, then one can compute (see \cite[Cor.\ 18]{kani5}) 
$$e([Q_L]) = [q_L]^2.$$  

Let us briefly hint at the moduli theoretic nature of this equality.  Kani's form $Q_L$ does not depend on any choice of product decomposition $A_m \cong E \times E'$, whereas our quadratic form relies on such a choice.  If $A_m \cong F \times F'$ is another such decomposition corresponding to a new quadratic form $q_L'$, then $[q_L]$ and $[q_L']$ differ by a $2$-torsion element in $\Pic(\O_{-4m})$; hence the class $[q_L]^2$ is independent of the choice of decomposition $A_m \cong E \times E'$.  Kani's quadratic form $Q_L$ is natural when studying $A_m$ inside the larger moduli space of principally polarized abelian surfaces; indeed, the definition of $Q_L$ applies to any such surface.  Our quadratic form $q_L$ only makes sense for split abelian surfaces, but is well suited for studying the surfaces $A_m$, as this paper shows.    


\subsection*{Acknowledgments}
The authors thank Ernest H.\ Brooks for some helpful discussions and Ernst Kani for some clarifying remarks.  The second author was partially supported by the National Science Foundation grant DMS-0943832.

\section{Preliminaries on abelian surfaces}\label{prelims}
Let $A$ be an abelian surface over an algebraically closed field $k$ of characteristic 0.  By definition, the N\'eron-Severi group of $A$ is $\NS(A) = \Pic(A)/\Pic^0(A)$.  Concretely, $\NS(A)$ is isomorphic to the group of divisors on $A$, modulo the subgroup of divisors in the kernel of the intersection pairing
$$\Div(A) \times \Div(A) \to \Z.$$  It is well known that $\NS(A)$ is a free abelian group of rank at most 4.  

We often abuse notation and consider divisors and line bundles as elements of $\NS(A)$.  In particular, if we are given an isomorphism $\phi: E \times E' \cong A$, with $E, E'$ elliptic curves, then we write $h$ and $v$ for the classes of $\phi(E \times \{0\})$ and $\phi(\{0\} \times E')$ in $\NS(A)$.  These are the horizontal and vertical axes \textit{with respect to the product decomposition $\phi$}, but $\phi$ will generally not be mentioned explicitly.    
\begin{definition}
An ample line bundle $L$ on $A$ is \textit{smooth} if $|L|$ contains a smooth (connected) curve.  A class $M \in \NS(A)$ is smooth if some line bundle $L$ in the class of $M$ is smooth.    
\end{definition}
\begin{lemma}\label{globgen}
If $L \in \Pic(A)$ is globally generated and ample, then $L$ is smooth.
\end{lemma}

\begin{proof}
By Bertini's theorem, $|L|$ contains a smooth but possibly reducible divisor $D$.  We may write $D = \sum_{i = 1}^n C_i$ with $C_i$ smooth curves and $C_i.C_j = 0$ for $i \neq j$.  If $n = 1$, then we are done.  Otherwise, the $C_i$ must be elliptic curves, for higher genus smooth curves on $A$ are automatically ample.  But then $C_i^2 = 0 $ and so $D^2 = 0$, which contradicts the fact that $L$ is ample.    
\end{proof}

\begin{proposition}\label{nakai}
Suppose $A= E \times E'$ is a product of two elliptic curves and let $h,v \in \NS(A)$ be the two axes.  Then $L\in \Pic(X)$ is ample if and only if $L.L > 0$ and $L.(h+v) > 0$.  
\end{proposition}

\begin{proof}
See \cite[4.3.3]{BL}.
\end{proof}

The following result, describing the group $\NS(A)$ for a product surface, is well-known.
\begin{proposition}\label{ns}
If $A = E \times E'$ is a product of elliptic curves, then the map 
$$\Z \oplus \Hom(E,E') \oplus \Z \to \NS(A)$$
$$(a, \l, b) \mapsto (a-1)h + \Gamma_\l +  (b - \deg(\l))v,$$
is an isomorphism of groups.  Here, $\Gamma_\l \subset E \times E'$ is the graph of $\l: E \to E'$.    
\end{proposition}

Recall that an ample line bundle $L \in \Pic(A)$ is \textit{of type} $(d_1,d_2)$ if the kernel of the isogeny $\phi_L: A \to \hat A$ is isomorphic to $(\Z/d_1\Z)^2 \times (\Z/d_2\Z)^2$, with $d_1$ dividing $d_2$. 

\begin{proposition}\label{crit}
\normalfont Let $A$ be an abelian surface and suppose $L \in \Pic(A)$ is ample of type $(d_1,d_2)$.  Then the following are equivalent.
\begin{enumerate}
\item $L$ is not smooth.
\item $d_1 = 1$ and there are elliptic curves $E,E'$ and an isomorphism of polarized abelian varieties $$(A,L) \cong (E \times E', h +d_2 v),$$
where $h$ and $v$ are the natural horizontal and vertical divisors on $E \times E'$.
\item There exists an elliptic curve $E \subset A$ such that $E.L = 1$.      
\end{enumerate}
\end{proposition}

\begin{proof}

(2) clearly implies (3) and (3) implies (2) by  \cite[5.3.13]{BL}. (2) implies (1) because any $C$ in $|L|$ visibly admits a non-constant map to an elliptic curve of degree 1.  So suppose that $|L|$ has no smooth curves.  Then we must have $d_1 = 1$, because otherwise $L$ is globally generated and hence has smooth curves by Lemma \ref{globgen}.  If $d_2 > 1$, then by \cite[10.1.1]{BL}, $L$ has a fixed component if and only if there is an isomorphism of polarized abelian surfaces as in (2).  So we may assume $L$ has no fixed component.   If $d_2 \geq 3$, then $L$ is again globally generated by \cite[10.1.2]{BL}.  If $d_2 = 2$, then $L$ has four base points, but the general member of $|L|$ is smooth by \cite[10.1.3]{BL}.  Finally, if $d_2 = 1$, then by Matsusaka-Ran \cite[11.8.1]{BL} either the single curve $C \in |L|$ is smooth or there is an isomorphism $A \cong E \times E'$ identifying $L$ with the canonical principal polarization on $E \times E'$.    
\end{proof}

\begin{theorem}[Reider]\label{reider}
Let $d \geq 5$ and let $L \in \Pic(A)$ be ample of type $(1,d)$.  Then $L$ is very ample if and only if there are no elliptic curves $E \subset A$ such that $E.L \leq 2$.  
\end{theorem}
\begin{proof}
See \cite[10.4]{BL}.
\end{proof}

The classes of elliptic curves in $\NS(A)$ can be described purely in terms of the intersection pairing:
\begin{proposition}[\cite{kani3}]\label{kani3}
A class $L \in \NS(A)$ is the class of an elliptic curve $E \subset A$ if and only if $L$ is indivisible in $\NS(A)$, $(L.L) = 0$, and $L.H > 0$ for some ample $H \in \NS(A)$.  
\end{proposition}

A consequence of Proposition \ref{crit} and Theorem \ref{reider} is that being smooth or very ample is a \textit{numerical} property for line bundles on an abelian surface.  Moreover, Proposition \ref{kani3} shows that the smooth and very ample classes are in fact determined by the intersection pairing on $\NS(A)$.        

\begin{definition}
A polarization $L \in \NS(A)$ is \textit{merely ample} if $L$ is not very ample.\footnote{Recall from the introduction that a polarization is ample by definition.}   
\end{definition}

The following is a more precise version of Reider's theorem for polarizations of odd degree.

\begin{theorem}\label{precisereider}
Let $A$ be an abelian surface and let $L \in \NS(A)$ be a polarization of odd degree $d \geq 5$.  If $L$ is smooth and merely ample, then there exist elliptic curves $E,F \subset A$ such that $E.L = 2$, $E[2] = F[2] = E\cap F$, and $\mu^*L \equiv 2dh + 2v \in \NS(E \times F)$, where  $\mu: E \times F \to A$ is the subtraction $4$-isogeny. 

Conversely, suppose $d \geq 1$, not necessarily odd.  If $E$ and $F$ are elliptic curves and $\phi: E[2] \to F[2]$ is an isomorphism of groups, then the quotient $B = E \times F / \Gamma_\phi$ admits a degree $d$ polarization $L \in \NS(B)$ such that $\pi^*L = 2dh+2v$.  Here, $\Gamma_\phi \subset E \times F$ is the graph of $\phi$ and $\pi: E \times F \to B$ is the canonical $4$-isogeny.  Moreover, $L.\pi(E \times 0) = 2$, so $L$ is not very ample.  
\end{theorem}

\begin{proof}
By Reider's theorem, $L.E = 2$ for some elliptic curve $E \subset A$.  Let $F$ be the complementary elliptic curve corresponding to $E$ with respect to the polarization $L$ (see \cite[\S5.3]{BL}) and write 
$$\mu: E \times F \to A$$ for the subtraction isogeny.  Since $L.E = h^0(L|_E) = 2$, one knows that $\ker \mu = E \cap F \subset E[2]$ \cite[5.3.11]{BL}.  

We claim that $\ker \mu = E[2]$.  We set $M = \mu^*L$ and note that $M = ah + 2v \in \NS(E \times F)$ for some integer $a$, by \cite[5.3.6]{BL}.  As $L$ has degree $d$, we must have $a = \frac{d}{2}\deg(\mu)$.  So if $\ker \mu \neq E[2]$, then $\deg(\mu) = 2$ and $\mu^*L = dh + 2v$.  But in order for $M$ to be in the image of $\mu^*: \NS(A) \to \NS(E \times F)$, we need $\ker (\mu) \subset K(M)$, where $K(M)$ is the kernel of $\phi_M : E \times F \to \widehat{E \times F}$  \cite[\S 23]{Mum}.  Since $K(M) = E[2] \times F[d]$ and $d$ is odd, this can only happen if $\mu$ is a map of the form
$$E \times F \to E/H \times F \cong A$$ for a subgroup $H \subset E[2]$ of order 2.  In that case, $L = d\tilde h + \tilde v$ with respect to the decomposition $A \cong E/H \times F$, so $L$ is not smooth, a contradiction.   So $\ker \mu  = E[2] = F[2]$ and $\mu^*L = 2dh + 2v$, as claimed.     

For the converse statement, we again set $M = 2dh + 2v$.  We need to show that there exists $L \in \NS(B)$ such that $\pi^*L= M$.  This is the case if and only if $\ker \pi \subset K(M)$ and $\ker \pi$ is isotropic for the Riemann form $e^M$ \cite[\S23]{Mum}.  As $K(M) = E[2] \times F[2d]$,which contains the 2-torsion on $E \times F$, the first condition is satisfied.  To prove the second condition, we choose a basis $P,Q$ for $E[2] \cong \F_2^2$ and compute 
$$e^M\left((P,\phi(P)), (Q,\phi(Q))\right) = e_2(P,Q)e_2(\phi(P),\phi(Q)) = (-1)^2 = 1,$$
showing that $\ker \mu$ is isotropic.    
\end{proof}

\begin{remark}
For $d$ odd, we can interpret Proposition \ref{crit} and Theorem \ref{precisereider} in terms of moduli problems as follows.   Let $\A_2(d)$ be the moduli stack of $(1,d)$-polarized abelian surfaces $(A,L)$, and let $Z \subset \A_2(d)$ be the locus of pairs $(A,L)$ with $L$ merely ample.  Then $Z$ has two components, one isomorphic to $Y(1) \times Y(1)$ and another admitting a finite \'etale map from the modular diagonal quotient surface $\GL_2(\Z/2\Z) \backslash Y(2) \times Y(2)$.  Here, $Y(N)$ is the moduli stack of elliptic curves with full level $N$ torsion structure.   
\end{remark}

\begin{remark}
If $d$ is even then the merely ample locus $Z \subset \A_2(d)$ has a third component  which admits a finite \'etale map from $Y_1(2) \times Y_1(2)$.  
\end{remark}

\begin{lemma}\label{vapullback}
Let $f: A \to B$ be an isogeny of abelian surfaces and $L \in \Pic(B)$.  If $L$ is smooth then $f^*L$ is smooth.  If $L$ is very ample bundle and $\deg(f^*L)$ is squarefree, then $f^*L$ is very ample. 
\end{lemma}

\begin{proof}
The very ampleness statement follows from Theorem \ref{reider}.  Indeed, it suffices to show that $f^*L.E > 2$ for any elliptic curve $E$ on $A$.  But $f_*E = nE'$ for some elliptic curve $E'$ on $B$ and so
$$(f^*L. E) = f_*(f^*L.E) = (L.f_*E) = n(L.E') > 2,$$
as desired.  A similar proof using Proposition \ref{crit} works for smoothness.
\end{proof}

\begin{lemma}\label{disjointEC}
If $A$ is an abelian surface and $E,F \subset A$ elliptic curves subgroups, then $E.F = 0$ if and only if $E = F$.  
\end{lemma}

\begin{lemma}\label{unique}
Let $A$ be an abelian surface and $L \in \NS(A)$ a polarization of degree $d \geq 5$.  Then there is at most one elliptic curve $E \subset A$ such that $L.E = 2$.
\end{lemma}

\begin{proof}
Suppose $F \subset A$ is another such elliptic curve.  By the Hodge index theorem, we have
$$4d(E.F) = L^2(E + F)^2 \leq (L.(E + F))^2 = 16,$$
so $d(E.F) \leq 4$.  This forces $(E.F) = 0$; hence $E = F$ by Lemma \ref{disjointEC}.  
\end{proof}

\begin{definition} Two polarizations $L, M \in \NS(A)$ are \textit{equivalent} if there exists $\alpha \in \Aut(A)$ such that $L = \alpha^*M$. 
\end{definition}
In the following definitions the term `polarizations' is used as an abbreviation for the phrase `equivalence classes of  polarizations'.  

\begin{definition}  For integers $d \geq 1$,
\begin{itemize}
\item $N(A,d)$ is the number of polarizations on $A$ of degree $d$.  
\item $N_\sm(A,d)$ is the number of \textit{smooth} polarizations on $A$ of degree $d$.
\item $N_\va(A,d)$ is the number of \textit{very ample} polarizations on $A$ of degree $d$.   
\end{itemize}
\end{definition}

\section{Product abelian surfaces of Picard number 2}\label{picrank2}
For completeness, we record in this section the (straightforward) answers to all our counting questions from the introduction, when $A = E \times E'$ is a product of non-isogenous elliptic curves $E$ and $E'$.  In this case, $\NS(A) = \Z h + \Z v$, where $h$ is the class of the horizontal divisor $E \times 0$ and $v$ is the class of the vertical divisor $0 \times E'$.  If $L \equiv ah + bv$, then the degree of $L$ is $d(L)=\frac{1}{2}L.L = ab$.  

\begin{proposition}
$L \equiv ah + bv \in \NS(A)$ is ample if and only if $a$ and $b$ are both positive. 
\end{proposition}

\begin{proof}
This follows from Proposition \ref{nakai}.
\end{proof}

Our next theorem computes $N(A,d)$, $N_\sm(A,d)$, and $N_\va(A,d)$ for $A=E\times E'$.
\begin{theorem}
\normalfont If $A$ is a product of two non-isogenous elliptic curves, then 
\begin{enumerate}
\item $N(A,d) =  \sigma_0(d)$.
\item $N_\sm(A,d) = \begin{cases} \sigma_0(d) - 2 &\mbox{if } d > 1\\ 0 &\mbox{if } d = 1.\end{cases}$  
\item $N_\va(A,d) =  \begin{cases} \sigma_0(d) - 4 &\mbox{$d \geq 5$ even}\\ \sigma_0(d) - 2 &\mbox{$d \geq 5$ odd.}\end{cases}$
\end{enumerate}
Here, $\sigma_0(d)$ is the number of divisors of $d$.  
\end{theorem}

\begin{proof}
Note that $\Aut(A) \cong \Z/2 \times \Z/2$ acts trivially on $\NS(A)$.    Moreover, the only elliptic curves on $A$ are translations of $h$ and $v$.  The theorem now follows immediately from Proposition \ref{crit} and Theorem \ref{reider}.     
\end{proof}

As an immediate consequence, we can determine for which $d$ there exist smooth or very ample polarizations on $A$ of degree $d$.

\begin{theorem}\label{thmPR2}
Let $d$ be a positive integer.  Then $A$ admits a smooth polarization of degree $d $ if and only if $d$ is composite. For $d\geq 5$, $A$ admits a very ample polarization of degree $d$ if and only if $d$ is neither a prime nor twice a prime.  
\end{theorem}

\begin{corollary}
There exists a smooth projective curve $C \subset A$ of genus $g = d + 1 \geq 2$ if and only if $d$ is composite.  
\end{corollary}

\begin{corollary}
The surface $A$ is not the Jacobian of a smooth genus $2$ curve.
\end{corollary}

\begin{corollary}
There is no embedding $A\hookrightarrow\P^4$.  
\end{corollary}

\begin{proof}
$A$ admits no very ample line bundles of degree 5.  
\end{proof}


\section{The correspondence in the non-CM case}\label{param}
In this section we let $E$ and $E'$ be isogenous elliptic curves without CM and set $A = E \times E'$.  By Proposition \ref{ns}, $\rho(A) = 3$.  Let $\lambda: E \to E'$ be a cyclic isogeny satisfying $\ker \lambda \cong \Z/m\Z$ for some $m \geq 1$.  Thus $\Hom(E, E')$ is simply $\Z\lambda$.  Then $\NS(A) \cong \Z h \oplus \Z v \oplus \Hom(E,E')$, with $h$ and $v$ the horizontal and vertical classes as before.  The inverse isomorphism sends $\lambda$ to $$X_\lambda := [\Gamma_\lambda] - h - mv.$$   
The class $X_\l$ is orthogonal to $h$ and $v$ and if $L \equiv ah + bX_\l + cv$, then the degree of $L$ is  
$$d = \frac{1}{2}(L.L) = ac - b^2m.$$
\begin{remark}
One can define $X_{\hat \l} \in \NS(A)$ analogously and we have $X_{\hat \l} = X_\l$.  So this choice of basis does not favor $E$ over $E'$.     
\end{remark}

Note that the determinant of the intersection pairing on $\NS(A)$ is equal to $2m$, so the integer $m$ is an invariant of the surface and does not depend on the choice of decomposition of $A_m$ as a product of elliptic curves.  We follow Kani and make the following definition.

\begin{definition}
An abelian surface $A$ is said to be of type $m$ if there is an isomorphism $A \cong E \times E'$ where $E, E'$ are elliptic curves admitting a cyclic isogeny $E \to E'$ of degree $m$.  
\end{definition}

We denote by $A_m$ \textit{any} abelian surface $A = E \times E'$ of type $m$ and Picard number 3.  This notation is convenient because we often consider abelian surfaces of different types at the same time.  Moreover,  all surfaces $A_m$ have isomorphic N\'eron-Severi groups (as quadratic spaces), so the numbers $N_*(A_m, d)$ only depend on $m$ and $d$, as the notation suggests.  We will rarely consider two non-isomorphic surfaces of the same type at the same time, so this should not cause confusion.  

Now set $A_m = E \times E'$ as before, and let $L \equiv ah + b X_\l + cv \in \NS(A_m)$ be an arbitrary class.

\begin{lemma}\label{ample}
$L$ is ample if and only if $d > 0$ and $a,c > 0$.  
\end{lemma}

\begin{proof}
This follows from Proposition \ref{nakai}.
\end{proof}

Associated with the class $L \equiv ah + bX_\l + cv\in\NS(A)$ is the quadratic form $$q_L : \Hom(E, A_m) \to \Z$$ $$f \mapsto \deg(f^*L).$$  Using the natural basis $\{(1,0), (0, \l)\}$ for 
$$\Hom(E,A_m) = \Hom(E, E) \oplus \Hom(E, E'),$$ 
one computes 
$$q_L([y], x\l) = amx^2 - 2bmxy + cy^2.$$  
Note that in order to define the quadratic form $q_L$, we need to choose a direct factor $E$ of $A$.  We have decided to hide this from the notation because we have fixed a decomposition $A = E \times E'$.   


\begin{notation}Let $V_m$ be the space of integral symmetric bilinear forms which are degenerate over $\Z/m\Z$.  Equivalently, $V_m$ is the space of quadratic forms $q = [A,2B,C]$ with $A,B \in m\Z$ and $C \in \Z$.\footnote{As usual, we write $[a,b,c]$ for the quadratic form $ax^2 + bxy + cy^2$.}  For $d \geq 1$, let $V_{m,d} \subset V_m$ be the set of positive definite $q$ of discriminant $-4md$.
\end{notation}

The set $V_m$ has a natural action of $\Gamma_0(m)$ by linear transformation of variable.  Here, $\Gamma_0(m)$ is the subgroup of $\GL_2(\Z)$ with lower left corner divisible by $m$.  This action preserves the subsets $V_{m,d}$.  We also note that $\Aut(A_m)$ is isomorphic to $\Gamma_0(m)$ via 
$$  \left( \begin{array}{cc}
a & b\hat\l \\
c\l & d \end{array}\right) \mapsto
 \left( \begin{array}{cc}
a & b \\
mc & d \end{array}\right).$$  



\begin{theorem}\label{bijnoncm}
The map $L \mapsto q_L$ is a bijection $\NS(A_m) \to V_m$ with the following properties:
\begin{enumerate}
\item $L$ is ample if and only if $q_L$ is positive definite.
\item The discriminant of $q_L$ is $-4m\deg(L)$
\item The correspondence is $\Gamma_0(m)$-equivariant 
\end{enumerate} 
In particular, for any $d \geq 1$ there is an induced bijection 
$$\Aut(A_m)\backslash \NS(A_m)_d^\amp \rightarrow \Gamma_0(m) \backslash  V_{m,d},$$
where $\NS(A_m)_d^\amp$ is the set of ample classes of degree $d$ in $\NS(A_m)$.  
\end{theorem}

\begin{proof}
The fact that $L \mapsto q_L$ is a bijection follows immediately from the explicit formula for $q_L$.  Property (1) follows from Lemma \ref{ample}; (2) and (3) are simple computations.  
\end{proof}

The set $\Gamma_0(m) \backslash V_{m,d}$ can be understood in terms of the more familiar $\GL_2(\Z)$-equivalence relation on quadratic forms.  Before explaining this, we first describe the elliptic curves on $A_m$ and determine the different decompositions of $A_m$ as a product of elliptic curves.  Such decompositions are in bijection with reducible principal polarizations, i.e. polarizations of the form $F + F'$ for two elliptic curves $F, F' \in \NS(A_m)$ such that $F.F' = 1$.  

For each $k$ dividing $m$, let $H_k$ be the unique subgroup of $\ker (\l : E \to E')$ of order $k$ and set $E_k = E/H_k$.

\begin{proposition}\label{redpp}
Any elliptic curve on $A_m$ is isomorphic to $E_k$ for some $k$ dividing $m$.  For each divisor $k$ of $m$ satisfying $(k,m/k) = 1$, there is an isomorphism $A_m \cong E_k \times E_{m/k}$, giving a reducible principal polarization on $A$.  Up to automorphisms of $A_m$, these are the only reducible principal polarizations. 
\end{proposition}  
\begin{proof}
Any elliptic curve on $A_m$ is isogenous to $E$.  It is therefore the image of a cyclic isogeny $f_{x,y}: E \to A_m = E \times E'$, where $x$ and $y$ are integers and where 
$$f_{x,y}(P) = (x(P), y\l(P)).$$  The kernel of $f_{x,y}$ is $H_k$, where $k = \gcd(x,m)$.  The proves the first part of the proposition.  

For the second part, note that $E_1 = E$ and $E_m = E'$.  The natural projections $E_{k_1} \to E_{k_2}$ for $k_1 | k_2$ form a lattice of isogenies corresponding to the lattice of divisors of $m$.  As the number $m$ is an invariant of the abelian surface, the only possible way to write $A_m \cong E_k \times E_j$ for some $j$ dividing $m$ is if $j = k' := m/k$ and $(k, k') = 1$.  For if $j \neq k'$ or $(k, k') > 1$, one checks (using dual isogenies if necessary) that $E_k$ and $E_j$ are connected by an isogeny of degree less than $m$.  

The last thing to check is that $A_m$ really is isomorphic to the product of $E_k \times E_{k'}$ when $(k,k') = 1$.  One can write down an isomorphism explicitly as follows.  Let $\l_k : E \to E_k = E/H_k$ be the natural projection, and let $\mu_k : E_k \to E'$ be the unique isogeny such that $\l = \mu_k \circ \l_k$.

Now let $r$ and $s$ be integers such that $rk - sk' = 1$ and consider the map $\phi: E \times E' \to E_k \times E_{k'}$ defined by
$$\phi: (P,Q) \mapsto (r\l_k(P) - \hat\mu_k(Q), \hat \mu_{k'}(Q) - s\l_{k'} (P)).$$
One checks that if $(P,Q) \in \ker\phi$, then $rkP = \hat\l(Q) = sk'P$.  As $rk-sk' = 1$, $P$ must be 0.  But then $Q$ is in the kernel of both $\hat\mu_k$ and $\hat\mu_{k'}$, and is therefore also 0.  So $\phi$ is an isomorphism.    
\end{proof}

Now we explain the connection with the more familiar class groups, but first some terminology.

\begin{definition}
An integral binary quadratic form $(A,2B,C)$ is \textit{matrix-primitive} if $\gcd(A,B,C) = 1$, i.e. if the corresponding symmetric bilinear form is primitive.  
\end{definition}
\begin{notation}
For any integer $D$, let $V^\mprim_{4D}$ be the space of matrix-primitive quadratic forms of discriminant $4D$.    
\end{notation}

\begin{remark}
If $D$ is odd, then we think of $V^\mprim_{4D} = V^\prim_{4D} \cup V^\prim_D$ as the union of the primitive quadratic forms of discriminant $4D$ and (twice) the primitive quadratic forms of discriminant $D$.  Note that the second set is empty if $D \equiv 3$ (mod 4).  If $D$ is even, then $V^\mprim_{4D} = V^\prim_{4D}$.     
\end{remark}

To prove the main theorems, we will carefully analyze the set $\Aut(A_m)\backslash \NS(A_m)_d^\amp$ for \textit{squarefree} values of $d$.  If $L \equiv ah + bX_\l + cv \in \NS(A_m)^\amp_d$, and $d = \deg(L) = ac - b^2m$ is squarefree, then $L$ is indivisible in $\NS(A)$, i.e. $\gcd(a,b,c) = 1$.  Notice that even though $L$ is indivisible in $\NS(A)$, $q_L$ may not be matrix-primitive.  But the matrix-content of $q_L$, that is $\gcd(am, bm, c) = \gcd(c,m)$, is a divisor of $\gcd(m,d)$.  Thus $\frac{1}{(c,m)}q_L$ is in $V^\mprim_{-4md/g^2}$.
\begin{proposition}\label{formsmap}
Let $d\geq 1$ be squarefree.  Then the map $L \mapsto \frac{1}{(c,m)}q_L$ induces a surjective map $$\Psi_{m,d}:  \Aut(A_m) \backslash \NS(A_m)_d^\amp \rightarrow \coprod_{g | (m,d)} \GL_2(\Z) \backslash V^\mprim_{-4md/g^2}.$$
The fiber above an element $[q] \in \GL_2(\Z) \backslash V^\mprim_{-4md/g^2}$ has size $|\Gamma_0(m)\backslash \Gamma_0(m/g)/\Aut(q)|.$

\noindent In particular, if $m$ and $d$ are coprime, then $\Psi_{m,d}: L \mapsto [q_L]$ is a bijection 
$$\Aut(A_m)\backslash \NS(A_m)^\amp_d \to \GL_2(\Z)\backslash V_{-4md}^\mprim.$$ 
\end{proposition}

\begin{remark}
Unless $q$ is 2-torsion in the class group, we have $\Aut(q) = \{\pm 1\}$, which acts trivially.  In this case, the size of the fibers of $\Psi_{m,d}$ is simply the index  of $\Gamma_0(m)$ in $\Gamma_0(g')$.  If $q$ is 2-torsion, $\Aut(q)/\{\pm1\}$ has size 2 whenever $q$ has discriminant $D < -4$.  
\end{remark}

\begin{proof}
It is not hard to see that $\Psi_{m,d}$ is surjective.  For instance, if $g= 1$, then given a matrix-primitive quadratic form $q$ of discriminant $-4md$, one can find a $\GL_2(\Z)$-equivalent form in $V_{m,d}$ by moving the double root of $q$ over $\Z/m\Z$ to 0, at least if $m$ is odd.  If $m$ is even then a little more care is required.  First recall that in this case, $q$ is automatically primitive. We can find an equivalent form with middle and outer coefficient divisible by $m$, but we need the middle coefficient divisible by $2m$.  This is automatic if $m$ is not divisible by 4.  If $m$ \textit{is} divisible by 4, and if $q$ has middle coefficient only divisible by $m$ and not $2m$, then $\gamma \cdot q  \in V_{m,d}$, where 
$$\gamma = \left( \begin{array}{cc}
1 & \frac{m}{2} \\
0 & 1 \end{array}\right).$$
This approach works for all $g | (d,m)$ by simply multiplying the corresponding matrix-primitive form in $V_{m/g,d/g}$ by $g$.    

To count preimages, one checks by a direct computation that if two forms in $V_{m,d}$ are $\GL_2(\Z)$ equivalent, then they are in fact $\Gamma_0(g')$-equivalent, where $g' = m/(c,m)$.  Now, the $\Gamma_0(m)$-equivalence classes which map to $q$ under $\Psi_{m,d}$ are indexed by the cosets of $\Gamma_0(m)$ in $\Gamma_0(g')$.  And one checks that two $\Gamma_0(m)$-classes collapse if and only if the corresponding cosets are in the same orbit of the automorphism group $\Aut(q) \subset \Gamma_0(g')$ of $q \in V_{m,d}$ acting on the coset space $\Gamma_0(m) \backslash \Gamma_0(g')$ on the right.         
\end{proof}

\begin{remark}
Here is a more geometric way to count the preimages of the map $\Psi_{m,d}$ above points where $g > 1$, which will be useful for us later on.  If $L \equiv ah + bX_\l + cv$ and $(c,m) = g$, then we consider the $g$-isogeny
$$f: A_m = E \times E' \stackrel{\l_g \times \id}\longrightarrow E_g \times E' =: A_{m/g},$$ where $E_g$ is the elliptic curve from Proposition \ref{redpp} and $\l_g : E \to E_g$ is the natural isogeny.  We choose the usual basis $\{h, X_{\mu_g}, v\}$ for $\NS(A_{m/g})$, where $\mu_g : E_g \to E'$ satisfies $\l = \mu_g \circ \l_g$.  One computes
$$f^*h = h, \hspace{4mm} f^*X_{\mu_g} = X_\l, \hspace{5mm} f^*v = gv.$$  Thus $L = f^*M$ for a polarization $M \in \NS(A_{m/g})$ of degree $d/g$.  In this way, you reduce to the case where $(d,m) = 1$.  But note that if $M, M' \in \NS(A_{m/g})$ are $\Aut(A_{m/g})$-equivalent, it is not in general true that $f^*M$ and $f^*M'$ are $\Aut(A_m)$-equivalent.  As $\Aut(A_{m/g}) = \Gamma_0(m/g)$, the discrepancy is measured exactly by the coset space $\Gamma_0(m)\backslash \Gamma_0(m/g)/\Aut(M)$.
\end{remark}

\bigskip
For squarefree $d$, Proposition \ref{formsmap} gives a formula for $N(A_m,d)$ in terms of class numbers of imaginary quadratic orders.  Recall that the $\GL_2(\Z)$-equivalence classes of primitive forms of discriminant $D< 0$ are in bijection not with the group $\Pic(\O_D)$ but with the set $\Cl(D)$ in which ideal classes and their inverses are identified.  Here, $\Pic(\O_D)$ is the class group of the quadratic order of discriminant $D$.  We write $h^+(D) = \#\Cl(D)$ and record the following formula for $N(A_m,d)$ in the case where $(m,d) = 1$.

\begin{corollary}\label{polcount}
If $(d,m) = 1$ and $d$ is squarefree, then 
$$N(A_m,d) = \begin{cases}
h^+(-4md) & md \mbox{ even}\\
h^+(-4md) + h^+(-md) & md \mbox{ odd}.
\end{cases}$$
Note that $h(-md) = 0$ when $md \equiv 1$ (mod 4).
\end{corollary}       

The bijection in Theorem \ref{bijnoncm} also gives a nice description of the $\Aut(A_m)$-equivalence classes of elliptic curves on $A_m$:

\begin{proposition}\label{ECs}
The map $L \mapsto q_L$ induces a bijection between the set of $\Aut(A_m)$-equivalence classes of elliptic curves on $A_m$ and the set of $\Gamma_0(m)$-equivalence classes of primitive integral binary quadratic forms of discriminant $0$.  The latter set is in bijection with the orbits of $\Gamma_0(m)$ acting on $\P^1(\Q)$.  If $k | m$, then the number of $\Aut(A_m)$-equivalence classes of elliptic curves isomorphic to $E_k = E/H_k$ is equal to the size of the group $(\Z/f_k\Z)^\times/\{\pm1\}$, where $f_k = \gcd(k,m/k)$.     
\end{proposition}

\begin{proof}
Let $E_{a,b}$ be the image of $f_{a,b} : E \to A_m$, as in the proof of Proposition \ref{redpp}.  Then $E_{a,b}$ is isomorphic to $E_k$ where $k = \gcd(a,m)$, and one computes
$$E_{a,b} \equiv (a^2/k)h + (ab/k)X_\l + b^2k'v \in \NS(A_m),$$
where $k' = m/k$.  Thus $q_{E_{a,b}}$ is the quadratic form $k'(ax -by)^2$ of discriminant 0.  The first two statements of the proposition are now clear once we note that the elliptic curves in $\NS(A_m)$ are exactly the indivisible classes $F$ such that $F.F =0$ and $h^0(F) > 0$ \cite{kani3}.  The last statement now follows from \cite[p. 234]{gz}; note though that our $\Gamma_0(m)$ contains elements of determinant $-1$.     
\end{proof}

The map $\Psi_{m,d}$ of Proposition \ref{formsmap} has an important extra equivariance property, which we elaborate on for the remainder of this section.  For each $k |m$, we write $k'$ for the complementary divisor $m/k$.  If $(k, k') = 1$, we may consider the endomorphism 
$$\e_k = \l_k \times \hat\mu_{k'}: A_m = E \times E' \to E_k \times E_{k'} \cong A_m.$$  Note that we have implicitly chosen an isomorphism $A_m \cong E_k \times E_{k'}$, so the map $\e_k$ is not very well defined.  In any case, it has degree $k^2$ and $w_k = \frac{1}{k} (\e_k)_*$ is an automorphism of the quadratic space $\NS(A_m)$.  In fact, 
\begin{equation}\label{atkin}
w_k(h) = h_k, \hspace{4mm} w_k(X_\l) = X_{\tilde \l}, \hspace{4mm} w_k(v) = v_k
\end{equation}
where $h_k, X_{\tilde\l}, v_k$ is the standard basis on $\NS(E_k \times E_{k'})$.   Straightforward computations result in the following lemma highlighting the connection between the collection of $w_k$ and the Atkin-Lehner involutions on $X_0(m)$.  We write $\omega(m)$ for the number of distinct primes dividing $m$.      
\begin{lemma}
The $w_k$ commute with each other, and satisfy $w_k^2 = 1$.  Moreover, up to $\Aut(A_m)$, the automorphism $w_k : \NS(A_m) \to \NS(A_m)$ is independent of the choice of decomposition $A_m = E \times E'$ and the choice of isomorphism $A_m \cong E_k \times E_{k'}$ above.  The $w_k$ therefore define a canonical action of $(\Z/2\Z)^{\omega(m)}$ on $\Aut(A_m)\backslash\NS(A_m)$.  
\end{lemma}  

For each $k |m$ and each $L \in \NS(A_m)$, define the quadratic form 
$$q_L^k : \Hom(E_k, A_m) \to \Z$$
$$f \mapsto \deg(f^*L).$$
We embed $\Hom(E_k,A_m)$ inside $\Hom(E, A_m)$ via $g \mapsto g \circ \l_k$, where $\l_k : E \to E_k = E/H_k$ is the natural map.  Then the restriction of $q_L$ to the subspace $\Hom(E_k, A_m)$ is equal to $kq_L^k$. Written in the standard basis $\{\hat {\l_k} \times 0, 0 \times \mu_k\}$ for $\Hom(E_k, E \times E')$, this means
$$q_L^k(y\hat \l_k \times  x\mu_k) = ak'x^2 - 2bmxy + cky^2,$$
i.e. $q_L^k = [ak', -2bm, kc].$  

\begin{lemma}\label{trans}
Let $L \in \NS(A_m)$ and suppose $k | m$ and satisfies $(k,k') = 1$.  Then $q^k_{w_k(L)}$ and $q_L$ are isomorphic quadratic forms.
\end{lemma}

\begin{proof}
This follows from (\ref{atkin}) and transport of structure.  
\end{proof}


If $g^2 | D$ and $D/g^2$ is a discriminant, then extension of ideals induces a surjective group homomorphism
$$e_g: \Pic(\O_D) \to \Pic(\O_{D/g^2}),$$ and hence a surjective map 
$$e_g: \Cl(D) \to \Cl(D/g^2).$$
If $q$ is a primitive quadratic form of discriminant $D$, then we write $[q]$ for the corresponding class in $\Pic(\O_D)$ or $\Cl(D)$.  We warn the reader now that we tend to conflate $\Cl(D)$ and $\Pic(\O_D)$ whenever there is no harm in doing so.     
 
Now suppose $(k,k') =1$ and $L \in \NS(A_m)$ has degree $d$, for some squarefree $d \geq 1$.  We write $g = \gcd(k,d)$.  For each $k$ we have the primitive quadratic form $f_k = \frac{1}{g}[k, 0, k'd]$ of discriminant $-4md/g^2$.  The class $[f_k]$ is 2-torsion, as can be seen from Lemma \ref{2torslemma}.  If, for example, $(m,d) = 1$, then the classes $[f_k] \in \Cl(-4md)$ form a subgroup isomorphic to $(\Z/2\Z)^{\omega(m)}$.  The following key proposition relates $q_L$ to $q_L^k$.  

\begin{proposition}\label{equivariance}
Let $L \in \NS(A_m)$ have degree $d \geq 1$ and suppose $q_L$ is primitive.  Also suppose $k | m$ satisfies $(k,k') = 1$ and write $g = (k,d)$.  Then 
$$e_g\left([q_L]\right) \cdot [f_k] = \left[\frac{1}{g}q_L^k\right] \in \Cl(-4md/g^2).$$ 
If $q_L$ is merely matrix-primitive, then 
$$e_g\left(\left[\frac{1}{2}q_L\right]\right) \cdot  [g_k] = \left[\frac{1}{2g}q_L^k\right] \in \Cl(-md/g^2),$$
where $g_k = \frac{1}{g}\left[k,k, (k + k'd)/4\right].$   
\end{proposition}

\begin{proof}
We only prove the first equation, the proof of the second being similar.  We first find a quadratic form representing $e_g(q_L)$.  Note that $(c,m) = 1$, since $q_L = [am,-2bm,c]$ is primitive.  Then as $ac - mb^2 = d$, we must have $(c,g) = 1$ and $g | a$.  We claim that 
$$e_g([q_L]) = [am/g^2, -2bm/g, c].$$
This follows from the following lemma.
\begin{lemma}\label{picardmap}
Let $g \geq 1$ be an integer.  If $q(x,y)$ is a primitive quadratic form of discriminant $D < 0$ and $h(x,y) = q(gx,y)$ is primitive of discriminant $g^2D$, then $e_g([h]) = [q].$
\end{lemma}
\begin{proof}
The quadratic form $q = [a,b,c]$ corresponds to the proper $\O_D$-ideal $I = a\Z + \left(\frac{-b + \sqrt{D}}{2}\right)\Z$, and $h = [ag^2,bg,c]$ corresponds to $g^2a \Z +  \left(\frac{-kb + k\sqrt{D}}{2}\right)\Z.$  The latter is equivalent to the ideal $I_g = ga\Z +  \left(\frac{-b + \sqrt{D}}{2}\right)\Z$ in $\Pic(\O_{Dg^2})$.  As $I$ is the $\O_D$-ideal generated by $I_g$, the lemma follows.  
\end{proof}
To prove the first equation in the proposition it now suffices to show that 
\begin{equation}\label{clgpeq}[am/g^2, -2bm/g, c] \cdot [k/g, 0, k'd/g] = [ak'/g, -2bm/g, ck/g]\end{equation}
in the class group.  We prove this using the old-fashioned definition of Gauss composition.  Actually, we will use Dirichlet's method of composition \cite{Cox}:  the product $[a,b,c] \cdot [a',b',c']$ of two primitive forms of discriminant $D$ is equal to 
$$[aa'/e^2 ,\,  B ,\, e^2(B^2 - D)/(4aa')],$$
where $e = \gcd(a,a', (b + b')/2)$ and $B$ is an integer satisfying 
$$B \equiv b \mbox{ (mod } 2a/e )$$
$$B \equiv b' \mbox{ (mod } 2a'/e)$$
$$B^2 \equiv D \mbox{ (mod } 4aa'/e^2).$$
In our case we have $D = -4md/g^2$, $e = k/g$, and $B = -2mb/g$.  Using this rule, we verify (\ref{clgpeq}) and finish the proof of the proposition.  
\end{proof}
 
\begin{corollary}\label{equiv}
Suppose $L$ is ample of degree $d$, $q_L$ is primitive, and $(k,k') = 1$.  Then
$$e_g(\Psi_{m,d}(L)) \cdot [f_k] = \Psi_{m,d}(w_k(L)) \in \Cl(4md/g^2).$$ 
If $L$ is merely matrix-primitive, then
$$e_g(\Psi_{m,d}(L)) \cdot [g_k] = \Psi_{m,d}(w_k(L)) \in \Cl(-md/g^2).$$
\end{corollary} 

\begin{proof}
Take $L$ to be $w_k(L)$ in Lemma \ref{trans}, and then use Proposition \ref{equivariance}.   
\end{proof} 
Let $W(m)$ be the group $(\Z/2\Z)^{\omega(m)}$.  Then Corollary \ref{equiv} essentially says that $\Psi_{m,d}$ is $W(m)$-equivariant.  This is not exactly true because the target of $\Psi_{m,d}$ does not quite have an action of $W(m)$, as the 2-torsion classes $[f_k]$ lie in different class groups.  However, in many situations we have actual $W(m)$-equivariance, as the following lemma shows:

\begin{corollary}\label{diagram}
Suppose $d$ is squarefree, $(m,d) = 1$, and $md \not \equiv 3$ $($mod $4)$.  Then $\Psi_{m,d}$ is equivariant for the action of $W(m)$, i.e. the following square commutes for each $k |m $ satisfying $(k,m/k) = 1$:
\[\begin{CD}
\Aut(A_m)\backslash \NS(A_m)_d^\amp @> w_k>> \Aut(A_m)\backslash \NS(A_m)_d^\amp\\
@V\Psi_{m,d}VV    @VV\Psi_{m,d}V  \\
\Cl(-4md) @> \times [f_k]>>  \Cl(-4md)
\end{CD}\]
\end{corollary}  
\begin{proof}
This follows from the previous corollary.  The condition on $md$ guarantees that for any ample $L$ of degree $d$, $q_L$ is primitive and not just matrix primitive and that the target of $\Psi_{m,d}$ is $\Cl(-4md)$ and not $\Cl(-4md) \coprod \Cl(-md)$.  
\end{proof}

\begin{remark}
When $md \equiv 3$ (mod 4), the square above still commutes if the top row is restricted to those $L$ such that $q_L$ is primitive.  If we restrict $\Psi_{m,d}$ to the $q_L$ which are merely matrix-primitive (so that the target of the vertical maps is actually $\Cl(-md)$), then $\Psi_{m,d}$ is equivariant with respect to the $[g_k]$ from Proposition \ref{equivariance}.    
\end{remark}
 
\section{The correspondence in the CM case}\label{singular} 

For this section, we let $A$ be an abelian surface of Picard number 4, also called a \textit{singular} abelian surface.  By \cite[4.1]{SM}, $A$ is isomorphic to $E_1 \times E_2$, where $E_i$ are elliptic curves with CM by the same imaginary field $K$, and we say that $A$ has CM by $K$.  Since $k = \bar k$ has characteristic 0, there is no loss in generality if we work over $\C$, which is what we will do.  

In what follows, we write $\O_f$ for the unique order of index $f$ in $\O_K$; $f$ is also called the \textit{conductor} of the order.  A lattice $\a$ in $K$ is a free $\Z$-module of rank 2 such that $\a \otimes_\Z \Q = K$.  The ring of multipliers $R(\a) = \{x \in K : x\a \subset \a\}$ is an order in $K$, and we write $f_\a$ for the conductor of $R(\a)$.  Two lattices $L, L'$ are equivalent if $L = \gamma L'$ for some $\gamma \in K$.  We write $[L]$ for the equivalence class of $L$.

\begin{theorem}[Mitani-Shioda]
The association 
\[([\a], f) \mapsto \C/\O_{ff_\a} \times \C/\a\] 
is a bijection from the set of pairs $([\a],f)$ of equivalence classes of lattices in $K$ and integers $f \geq 1$ to the set of isomorphism classes of singular abelian surfaces with CM by $K$.   
\end{theorem}

\begin{proof}
Let $A = E_1 \times E_2$ have CM by $K$.  Let $f_1$ and $f_2$ be such that $\End(E_i) \cong \O_{f_i}$.  We may then write $E_i \cong \C/\a_i$, where $\a_i$ is a lattice such that $R(\a_i) = \O_{f_i}$, i.e. $\a_i$ is a proper $\O_{f_i}$-ideal.  The equivalence class of $\a_i$ is uniquely determined by $E_i$.  There are in general many choices for the elliptic curves $E_i$ in the product decomposition, but the following lemma provides a particularly convenient decomposition. 

\begin{lemma}
We may choose $E_1$ and $E_2$ so that $f_2$ divides $f_1$ and $E_1$ is an arbitrary elliptic curve with $\End(E_1) = \O_{f_1}$.  Once the choice of such $E_1$ is made, $E_2$ is uniquely determined up to isomorphism.        
\end{lemma}

\begin{proof}
This follows from \cite[4.5]{SM} as we explain now (see also \cite[Cor. 66]{kaniCM}).  Let $L_1$ and $L_2$ be two lattices in $K$ with conductors $c_1$ and $c_2$.  The product $L_1L_2$ is itself a lattice and $R(L_1L_2)$ has conductor $\gcd(c_1,c_2)$.  Equivalence classes of lattices $L$ having ring of multipliers $R$ form a group under multiplication.  The content of \cite[4.5]{SM} is that 
$$\C/\a_1 \times \C/\a_2 \cong \C/L_1 \times \C/L_2$$ 
if and only if 
$$[\a_1\a_2] = [L_1L_2] \hspace{4mm} \mbox{ and }\hspace{4mm} \lcm(f_1,f_2) = \lcm(c_1,c_2).$$
Now choose $L_1$ to be any lattice of conductor $\lcm(f_1,f_2)$ and $L_2$ a lattice of conductor $(f_1,f_2)$.  Then $L_1L_2$ has conductor $(f_1,f_2)$ and so there exists a lattice $L'$ of conductor $(f_1,f_2)$ such that  $L_1L_2L' \sim \a_1\a_2$.  It follows that $A \cong \C/L_1 \times \C/L_2L'$ and this decomposition satisfies the desired divisibility property.        
\end{proof}

For any singular $A$, we can now write $A = E_1 \times E_2$ as above, with $f_2$ dividing $f_1$ and $\a_1 = \O_{f_1}$.  If we let $\a = \a_2$, and $f = f_1/f_2$, then we have just set up a bijection between isomorphism classes of singular abelian surfaces and pairs $([\a],f)$ where $[\a]$ is an equivalence class of lattices in $K$ and $f \geq 1$ is an integer.\footnote{This latter set is in bijection with $\SL_2(\Z)$-equivalence classes of positive definite binary quadratic forms, given by sending $([\a], f)$ to $fq_\a$, where $q_\a(\alpha) = \frac{\Nm(\alpha)}{\Nm(\a)}$ is the norm form on $\a$.  This recovers the version of this theorem found in \cite{SM}, where they show that the transcendental lattice of $A$ is isomorphic to $fq_\a$.   However, one can prove these results purely algebraically (and in any characeteristic) without use of Hodge theory, as is done in \cite{kaniCM}. }  
\end{proof}

For the rest of this section we let $A = A_{\a,f} = E_1 \times E_2$ be the singular abelian surface corresponding to the pair $([\a], f)$.  To ease notation we write $\O$ for $\O_{ff_\a}$, so that $E_1 = \C/\O$ and $E_2 \cong \C/\a$.  We also write $\O_\a$ for $R(\a)$.       
We have   
\[\NS(A) \cong \NS(E_1)  \oplus \Hom(E_1,E_2) \oplus \NS(E_2),\]
where $\Hom(E_1, E_2)$ embeds inside $\NS(A)$ via 
\[ \beta \mapsto X_\beta := \Gamma_\beta - h -  \deg(\beta) v\]
Thus an arbitrary $L \in \NS(A)$ can be written as $L \equiv av + X_\beta + ch$ for integers $a$ and $c$ and $\beta \in \Hom(E_1,E_2)$; recall that $v$ and $h$ are the vertical and horizontal divisors.  One checks easily that $\Hom(E_1,E_2)$ is orthogonal to both $v$ and $h$. There is also the formula 
$$(X_\beta. X_\gamma) = -\Tr(\hat\gamma\beta).$$  
It follows that the degree $d$ of $L$ is $\frac{1}{2}(L.L) = ac - \deg(\beta)$. 

As in the previous section, we attach to each $L \in \NS(A)$ the (rank 4) integral quadratic space
\[q_L : \Hom(E_1,A) \to \Z\]
\[q_L(g) = \deg(g^*L).\]
In this case, however, there is even more structure: $q_L$ endows $\Hom(E_1,A)$ with the structure of an $\O$-hermitian module.  To be more specific, we may think of $\Hom(E_1,A)$ as an $\O$-module under precomposition, and then endow this module with the following $\O$-hermitian pairing
\[\langle , \rangle_L : \Hom(E_1,A) \times \Hom(E_1,A) \to \O\] 
\[ \langle f, g\rangle_L =  \phi_{E_1}^{-1}\hat g \phi_L f \in \End(E_1)\cong \O. \]
Here, $\phi_L: A \to \hat A$ is the map induced by $L$ and $\phi_{E_1} : E \to \hat E$ the isomorphism induced by the canonical principal polarization on $E_1$.  
Note that 
\[\langle f,f\rangle_L =  \phi_{E_1}^{-1} \phi_{f^*L} = [\deg(f^*L)]  = q_L(f).\]
As an $\O$-module, we have 
\[\Hom(E_1,A) \cong  \End(E_1) \oplus \Hom(E_1,E_2) \cong \O \oplus \a.\]
We will refer to an $\O$-hermitian pairing on $\O \oplus \a$ as an $(\O,\a)$-hermitian form.  If $L = av + X_\beta + ch$ and $(x,y) \in \End(E_1) \oplus \Hom(E_1,E_2)$, then we compute 
\begin{equation}\label{hermform}
q_L(x, y) = a\deg(x) - \Tr(\hat y \beta x) + c\deg(y)
\end{equation}
Conversely, any $\O$-hermitian pairing on $\Hom(E,A)$ is of the above form, for certain $a,c \in \Q$ and 
\[\beta \in \Hom(E_1,E_2)_\Q \cong \a \otimes_\Z \Q \cong K.\]  
We say such a form is \textit{integral} if $c$ is an integer.\footnote{Note that $a$ is forced to be an integer and we show below that $\beta$ is forced to lie in $\Hom(E_1,E_2)$.}  

\begin{theorem}\label{hermbij}
The association $L \mapsto \langle \, , \, \rangle_L$ is a bijection between $\NS(A)$ and the space of integral $(\O,\a)$-hermitian forms.  Moreover:
\begin{enumerate}
\item $L$ is ample if and only if $q_L$ is positive definite.  
\item If $L$ has degree $d$ then $q_L$ has determinant $d$.  
\item $L \mapsto \langle \, , \, \rangle_L$ is $\Aut(A)$-equivariant, with $\Aut(A)$ acting on $\Hom(E_1,A)$ by composition. 
\end{enumerate}
\end{theorem}

\begin{proof}
$L$ is ample if and only if $(L.L) > 0$ and $L.(h+v) > 0$ if and only if $a$, $c$, and $d$ are all positive.  As with binary quadratic forms, this is the case if and only if $q_L$ is positive definite.  Property (2) is a straightforward computation.  Property (3) follows from the fact that
\[q_{\alpha^*L}(g) = \deg(g^*\alpha^*L) = \deg((\alpha g)^*L) = q_L(\alpha g).\]  

The last thing to prove is that every integral binary $\O$-hermitian pairing on $V = \Hom(E_1,A) \cong \O \oplus \a$ arises from some $L \in \NS(A)$.  Since the coefficients $a,c$ in equation (\ref{hermform}) are integral, it suffices to check that $\beta \in \a$ for any $(\O,\a)$-hermitian form.  Using our identifications $E_1 = \C/\O$ and $E_2 = \C/\a$,  the dual of the multiplication by $v \in \a$ map $\C/\O \to \C/\a$ is the map $\hat v: \C/\a \to \C/\O$, given by multiplication by $\frac{f}{\Nm(\a)}\bar v$.  Thus $\beta$ needs to satisfy $\beta\frac{f}{\Nm(\a)}\bar \a \subset \O$, i.e. $\beta \in \a$.    
\end{proof}

\begin{lemma}
For $A$ as above, we have 
\[\End(A) \cong \End_\O(\O\oplus \a) =   \left(\begin{array}{cc}\O & f \a^{-1}\\ \a&\O_\a \end{array}\right).  \]
and hence
\[\Aut(A) \cong \left\{ M  \in \left(\begin{array}{cc}\O & f \a^{-1}\\ \a&\O_\a \end{array}\right) : \det M \in \O_\a^\times\right\} \]
\end{lemma}     

\begin{proof}
The lattice in $\C^2$ defining $A$ is $\O \oplus \a$, so the first isomorphism is clear.  An element of $\End(A)$ can be written as a matrix 
\[\left(\begin{array}{cc} \alpha & \beta \\ \gamma & \delta \end{array}\right)\]
with $\alpha \in \End(E_1)$, $\beta \in \Hom(E_2, E_1)$, $\gamma \in \Hom(E_1,E_2)$, $\delta \in \End(E_2)$.  We may identify $\Hom(E_2,E_1)$, for example, with the lattice  
$$(\O : \a) = \{x \in K: x\a \subset \O\}.$$
For any lattice $L_1$, $L_2$ of conductors $c_1$ and $c_2$, one has the general formula \cite[Lem. 15]{kaniCM}
$$(L_1 : L_2) = \frac{c_1}{(c_1,c_2)} L_1L_2^{-1}.$$
Setting $f = f_1/f_2$, it follows that $\alpha \in \O$, $\beta \in (\O : \a) = f \a^{-1}$, $\gamma \in \a$, and $\delta \in \O_\a$, i.e.
$$\End(A) = \End(\O\oplus \a) =   \left(\begin{array}{cc}\O & f \a^{-1}\\ \a&\O_\a \end{array}\right).$$
\end{proof}

Let us write $\Gamma = \Gamma(\a,f)$ for $\GL_\O(\O \oplus \a)$, so that we can identify $\Aut(A)$ with $\Gamma$ as above.  Moreover, $\Gamma$ acts naturally on the space of integral $(\O,\a)$-hermitian forms. 

\begin{corollary}
The number $N(A,d)$ of polarizations on $A$ of degree $d$ is equal to the number of $\Gamma$-equivalence classes of integral positive definite $(\O,\a)$-hermitian forms of discriminant $d$.  
\end{corollary}

\begin{remark}
Just as in the non-CM case, isomorphism classes of $\O$-hermitian modules of fixed discriminant are in bijection with class groups of certain rings.  It is not surprising  (from the point of view of quadratic forms and also the point of view of moduli of abelian surfaces) that the relevant rings in the CM case are orders in definite quaternion algebras over $\Q$.  \textit{Unlike} in the non-CM case, these quaternionic class numbers can be computed explicitly using methods of Eichler.  The hermitian correspondence  we allude to goes back to Brandt and Lattimer in certain special cases, but a complete correspondence has not been fleshed out in the literature.  We plan to elaborate on this in a separate paper and to count (smooth, very ample, etc.) polarizations on singular abelian surfaces.         
\end{remark}

\begin{proposition}
The set of $\Aut(A)$-equivalence classes of elliptic curves on $A$ is in bijection with the orbits of $\Gamma$ acting by linear transformation of variable on $\P(\O \oplus \a) = \P^1(K)$.  
\end{proposition}
\begin{proof}
By Theorem \ref{hermbij} and Proposition \ref{kani3}, classes of elliptic curves are in bijection with $\Gamma$-equivalence classes of non-negative primitive $(\O,\a)$-hermitian forms of discriminant 0.  After choosing a basis, such forms take the shape $\Nm(v\alpha  - u \gamma)$, for some coefficients $u \in \O$ and $v \in \a$ and variables $(\alpha, \gamma) \in \O \oplus \a$.  Thus $\Ell(A)$ is in bijection with $\Gamma$-orbits of $\P(\O \oplus \a) = \P^1(K)$.  Explicitly, the point $P = (u:v)$ corresponds to the elliptic curve $E_P$ which is the image of the morphism $\phi_{u,v}: \C/\O \to \C/\O \times \C/\a$ given by $z \mapsto (uz, vz)$. 
\end{proof}

\begin{example}
Consider the case $f = 1$ and $\a$ is an $\O_K$-ideal.  So $A = E_1 \times E_2$ is a product of two elliptic curves with complex multiplication by $\O_K$.  Then there are $h_K$ different $\Aut(A)$-equivalence classes of elliptic curves on $A$, where $h_K$ is the class number of $K$.  If $\a = \b^2$ is a square in the class group, then $A$ is isomorphic to $\C/\b \times \C/b$, hence $\Gamma \cong \GL_2(\O_K)$.  The fact that $\#\GL_2(\O_K) \backslash \P^1(K) = h_K$ is due to Bianchi.  If $\a$ is not a square in the class group, then we adopt a different approach which should apply more generally.  One first notes that any elliptic curve $E \subset A$ necessarily has CM by $\O_K$, and similarly for the quotient $A/E$.  Thus, such an elliptic curve determines an extension
\[0 \to \b \to \O\oplus \a \to \c \to 0\]
of $\O_K$-modules, where $E \cong \C/\b$ and $[\c] = [\a\b^{-1}]$ in $\Pic(\O_K)$.  Conversely, any extension of this form gives rise to an elliptic curve $\C/\b \subset A$, and two elliptic curves on $A$ are $\Aut(A)$-equivalent if and only if the corresponding extensions give the same class in $\Ext^1_{\O_K}(\c,\b)$.  But $\Ext^1_{\O_K}(\c,\b) = 0$ since $\b$ and $\c$ are projective, so there is a single elliptic curve for each ideal class $\b$.  
\end{example}

In the general case of arbitrary $f \geq 1$ and lattice $\a$, there should be a reasonable formula for $\#\Gamma(\a,f) \backslash \P^1(K)$ in terms of class numbers of certain orders, analogous to Proposition \ref{ECs}.  But for an arbitrary order $\O$ and arbitrary torsion-free $\O$-ideals $\b,\c$, the group $\Ext^1_\O(\b,\c)$ is non-zero, which complicates the approach given in the previous example.

\section{Smooth polarizations on $A_m$}\label{smooth}

In this section we return to the non-CM case discussed in Section \ref{param} and determine the number $N_\sm(A_m,d)$ of smooth polarizations on $A_m$ of degree $d$ and also give a proof of Theorem \ref{mainsmooth}.  By Corollary \ref{polcount}, it suffices to identify the non-smooth (or reducible) polarizations of degree $d$.  By Proposition \ref{crit}, this is equivalent to counting the product polarizations on $A_m$ of the form 
$$L = dh + v$$ with respect to some product decomposition $A_m = E_1 \times E_2$ of $A_m$. 
\begin{corollary}\label{nonsm}
Let $\omega(m)$ be the number of distinct primes dividing $m$.  Then the number of non-smooth polarizations on $A_m = E\times E'$ of degree $d$ is equal to
$$\begin{cases}
1 & \mbox{if } m =1\\
2^{\omega(m) - 1} & \mbox{if } m > d = 1\\
2^{\omega(m)} & \mbox{if } m > 1 \mbox{ and } d > 1.
\end{cases}$$
\end{corollary}

\begin{proof}
This follows from Proposition \ref{redpp}.  The non-smooth polarizations of degree $d$ are the product polarizations as in Proposition \ref{crit} coming from the $2^{\omega(m)}$ different product decompositions $A_m \cong E_k \times E_{k'}$.  If $d = 1$, then the order of the product does not matter, giving half as many polarizations.   
\end{proof}

The corollary above simply counts the non-smooth polarizations. We can be more precise and identify the image of the non-smooth polarizations under the map $\Psi_{m,d}$.   Unless otherwise stated, we assume for the rest of the section that $d$ is squarefree.

\begin{proposition}\label{2t}
Let $L_k$ be the product polarization $d h_k + v_k$, where $h_k, v_k$ are axes with respect to a decomposition $A_m \cong E_k \times E_{k'}$.  Then under the composition
$$\Psi_{m,d} : \Aut(A_m)\backslash \NS(A_m)_d^\amp \longrightarrow \Gamma_0(m)\backslash V_{m,d} \to \coprod_{g | (m,d)} \GL_2(\Z) \backslash V^\mprim_{-4md/g^2}.$$
$L_k$ is sent to the $2$-torsion class $[f_k] = \frac{1}{g}[k,0,k'd] \in \Cl(-4md/g^2)$, where $g = (k,d)$.  
\end{proposition}

\begin{proof}
Up to $\Aut(A_m)$-equivalence, we have $L_k = dh_k + v_k = w_k(L)$, where $L = L_1 = dh + v$.  Now, $q_L = [md, 0,1]$ is primitive and represents 1, so $\Psi_{m,d}(L)$ is the trivial class in $\Cl(-4md)$ (see Lemma \ref{reps}).  By Corollary \ref{equiv}, we have
$$\Psi_{m,d}(L_k) = \Psi_{m,d}(w_k(L)) = e_g(\Psi_{m,d}(L))\cdot [f_k] = [f_k],$$
as desired.    
\end{proof}

Theorem \ref{bijnoncm}, Proposition \ref{formsmap} and Corollary \ref{nonsm} give an explicit formula for $N_\sm(A_m,d)$ when $d$ is squarefree.  For general squarefree $d$ the formula is a little complicated, so we write down the formula only for $d$ prime to $m$. 

\begin{corollary}
Let $A_m = E \times E'$ with $E \to E'$ a cyclic isogeny of degree $m$ and $\End(E) = \Z$.  Suppose $d > 1$ is squarefree and prime to $m$.  Then
$$N_\sm(A_m, d) =  \begin{cases}
h^+(-4md) - 2^{\omega(m)}  &  md \mbox{ even}\\
h^+(-4md) + h^+(-md) - 2^{\omega(m)} & md \mbox{ odd},
\end{cases}$$
and if $m > 1$, then
$$N_\sm(A_m, 1) =  \begin{cases}
h^+(-4m) - 2^{\omega(m) - 1}  &  m \mbox{ even}\\
h^+(-4m) + h^+(-m) - 2^{\omega(m) - 1} & m \mbox{ odd}.
\end{cases}$$
Here $h^+(D) = \#\Cl(D)$.  
\end{corollary}

Restricting to primes $d$ such that $(m,d) = 1$, we have the following more useful result.
\begin{corollary}\label{dprime}
Let $A_m = E \times E'$ with $E \to E'$ a cyclic isogeny of degree $m$ and $\End(E) = \Z$.  Suppose $d$ is a prime not dividing $m$ or that $d = 1$.  Then
$$N_\sm(A_m, d) =  \begin{cases}
0 & m = d= 1\\
\frac{1}{2}\left[h(-4md) + h(-md)\right] & md \mbox{ odd and } md > 1\\
\frac{1}{2} h(-4md) &  8 | md \\
\frac{1}{2}\left[ h(-4md) - h_2(-4md)\right] & md \mbox{ even, not divisible by $8$}.
\end{cases},$$
where $h(D)$ is the size of the class group $\Pic(\O_D)$ of discriminant $D$ and $h_2(D) = \#\Pic(\O_D)[2]$.  
\end{corollary}

\begin{proof}
This follows from a simple computation using the previous corollary and the following classical fact \cite[Prop. 3.11]{Cox}.

\begin{proposition}\label{2torsion}
Let $D\equiv 0,1$ $(\mod 4)$ be negative, and let $r$ be the number of odd primes dividing $D$.  Define $\mu = r$ if $D$ is odd and if $D = -4n$, with $n > 0$, set
$$\mu = \begin{cases} 
r &\mbox{if } n \equiv 3 \,(\mod 4)  \\ 
r+1 & \mbox{if } n \equiv 1,2 \,(\mod 4)\\
r+1 & \mbox{if } n \equiv 4 \,(\mod 8)\\
r+2 & \mbox{if } n \equiv 0 \,(\mod 8). 
\end{cases}$$
Then $\#\Cl(D)[2] =2^{\mu - 1}$. 
\end{proposition}   
\end{proof}

We can now prove Theorem \ref{smcurves}, which determine the set of integers $d$ such that $N_\sm(A_m,d) \neq 0$, i.e. such that $A_m$ contains a smooth curve of genus $d + 1$.     


\begin{proof}[Proof of Theorem $\ref{smcurves}$]
If $d$ is composite then $N_\sm(A_m,d) > 0$.  Indeed, if $d = pq$ is a factorization, then $ph + qv$ is globally generated, hence smooth by Lemma \ref{globgen}, and of degree $d$.  So we may assume $d$ is prime or equal to 1.  If $(m,d) = 1$, then Corollary \ref{dprime} gives the result.  Now assume that $d = p$ is prime and divides $m$.  We show case by case that there is at least one smooth line bundle $L \in \NS(A_m)$ of degree $p$.  
Since $p$ divides $m$, the image of $\Psi_{m,d}$ is $V^\mprim_{-4md} \cup V^\mprim_{-4m/d}$.  The $2^{\omega(m)}$ non-smooth polarizations always land in either $V^\prim_{-4md}$ or $V^\prim_{-4m/d}$.  

\begin{enumerate}
\item Case: $d = 2$.  In this case $m$ is even and the target of $\Psi_{m,d}$ is $V^\mprim_{-8m} \cup V^\mprim_{-2m}$. 
\begin{enumerate}
\item Subcase: $m$ divisible by 4.  Then $\#\Cl(-8m)[2] = 2^{\omega(m)}$.  Since $V^\mprim_{-2m}$ is non-empty and $\Psi_{m,d}$ is surjective, we conclude $N_\sm(A,d) > 0$.   
\item Subcase:  $-m/2 \equiv 1$ (mod 4).  The image of $\Psi_{m,d}$ is then $V^\prim_{-8m} \cup V^\prim_{2m} \cup V^\prim_{-m/2}$.  Since $V^\prim_{-m/2}$ is not empty and contains only smooth polarizations, $N_\sm(A,d) > 0$.  
\item Subcase: $-m/2 \equiv 3$ (mod 4).  The image of $\Psi_{m,d}$ is $V^\prim_{-8m} \cup V^\prim_{-2m}$ in this case.  We have $\#\Cl(-8m)[2] = 2^{\omega(m) - 1}$ and $\#\Cl(-2m)[2] = 2^{\omega(m) -1}$.  So there are $2^{\omega(m)}$ images under $\Psi_{m,d}$, which is the same as the number of non-smooth polarizations.  But the fibers of $\Psi_{m,d}$ above elements of $q \in V^\prim_{-2m}$ have size $|\Gamma_0(m)\backslash \Gamma_0(m/2)/\Aut(q)|$.  When $m > 2$, $\Aut(q)$ acts on the three element set $\Gamma_0(m)\backslash \Gamma_0(m/2)$ by a quotient of size at most 2, so the fibers are not singletons and $N_\sm(A_m,d) > 0$.  

If $m = 2$, then consider the line bundle $L = 2h + X_\l + 2v$.  Neither $q_L$ nor $q_L^2$ are primitive, hence they cannot represent 1.  Since $E$ and $E_2 = E'$ are the only elliptic curves on $A_2$ up to abstract isomorphism, this shows that $L.F > 1$ for any elliptic curve $F$ on $A_2$.  So $L$ is smooth and $N_\sm(A_2,2) > 0$.
\end{enumerate}  
\item Case: $d > 2$.  
\begin{enumerate}
\item Subcase: $m$ odd.  Then $\Cl(-md)$ is non-empty if $m$ and $d$ are different (mod 4), in which case $N_\sm(A,d) > 0$.  If $m \equiv d$ (mod 4), then $\#\Cl(-4md) = 2^{\omega(m) + 1} > 2^{\omega(m)}$, so $N_\sm(A,d) > 0$.  
\item Subcase: $m$ even.  Then $\#\Cl(-4md)[2]$ is at least $2^{\omega(m)}$, which is the number of non-smooth polarizations.  Since $\Cl(-4m/d)$ is non-empty, $N_\sm(A,d) > 0$.  
\end{enumerate}
\end{enumerate}   
\end{proof}

\section{Very ample polarizations on $A_m$}\label{vasect}
%
%
%

Before proving Theorem \ref{main}, we record some basic facts about quadratic forms and ideal classes.

\begin{lemma}\label{2torslemma}
Primitive quadratic forms $[a,ab,c]$ of discriminant $D$ correspond to $2$-torsion ideal classes in $\Pic(\O_D)$.
\end{lemma}
\begin{proof}
Recall that a primitive form $[a,b,c]$ of discriminant $D$ corresponds to the class of the ideal $\Z a +  \Z \frac{-b + \sqrt{D}}{2}$ in the quadratic ring of discriminant $D$.  Moreover, conjugation of ideals induces inversion on the class group.  As the ideal $(a, \frac{-ab + \sqrt D}{2})$ is fixed by conjugation, it follows that the corresponding ideal classes are 2-torsion.   
\end{proof}

\begin{lemma}\label{reps}
Let $f = [a,b,c]$ be a primitive positive definite quadratic form of discriminant $D$ and let $[\a]$ be the corresponding class of proper $\O_{D}$-ideals.  Then $[\a]$ contains an ideal of norm $a$.  
\end{lemma}

\begin{lemma}\label{kernel}
Let $D$ be a quadratic discriminant and $g \geq 1$ an integer.  Then the kernel of the natural map
$\Pic(\O_{g^2D}) \to \Pic(\O_D)$ has size 
$$\frac{g}{\left[\O^\times_D : \O_{g^2D}^\times\right]} \prod_{p | g} \left( 1 - \left(\frac{D}{p}\right) \frac{1}{p}\right).$$ 
\end{lemma}

\begin{proof}
See \cite{Cox}.  
\end{proof}

The next few propositions count the number of (equivalence classes of) smooth but merely ample line bundles on $A_m$
\begin{proposition}\label{oddm}
Let $d \geq 5$ be odd and squarefree and let $m \geq 1$ be odd.  Then there are $2^{\omega(m)}$ $\Aut(A_m)$-equivalence classes of smooth and merely ample line bundles of degree $d$ on $A_m$.  
\end{proposition}

\begin{proof}
First we construct $2^{\omega(m)}$ such line bundles.  For each $k | m$ such that $(k,k') = 1$, we choose an isomorphism $E_k \times E_{k'} \cong A_m$.  We let $\l_k: E_k \to E_{k'}$ be the usual minimal isogeny of degree $m$.  Since $m$ is odd, $\l_k$ induces an isomorphism $E_k[2] \cong E_{k'}[2]$.  We let $G_k \subset E_k \times E_{k'}$ be the subgroup 
$$\{ (P,\l_k(Q)): P \in E_k[2]\},$$
which is abstractly isomorphic to $(\Z/2\Z)^2$.  In fact, $G_k = \Gamma_{\l_k}[2]$, where $\Gamma_{\l_k}$ is the graph of $\l_k$. Since $$E_k \times E_{k'} \cong \Gamma_{\l_k} \times E_{k'},$$
the quotient $E_k \times E_{k'}/G_k$ is isomorphic to $A_m$.  Write $\pi: E_k \times E_{k'} \to A_m$ for the induced 4-isogeny.  Then by Theorem \ref{precisereider}, there exists a merely ample $L_k \in \Pic(A_m)$ such that $\pi^*L_k =  2dh_k + 2v_k$, where $h_k$ and $v_k$ are the axes with respect to the decomposition $E_k \times E_{k'}$.  

To prove that $L_k$ is smooth, we may assume $k = 1$ (we could always relabel the elliptic curves on $A_m$).  If $L_k$ is not smooth, then there exists a decomposition $A_m \cong F \times F'$ such that $L \equiv dh + v$, where $h$ and $v$ are axes with respect to this new decomposition.  But then $(dh + v. E) = 2$, which forces $E.h = 0$ and hence $F = E$.  Then $q_L = [dm, 0, 1]$ which does not represent 2, contradicting $L.E = 2$.  This shows the existence of smooth but merely ample $L_k$ for each of the $2^{\omega(m)}$ divisors $k |m$ satisfying $(k,m/k) =1$.  The $L_k$ have the property that $L_k.F =2$, for some elliptic curve $F$ isomorphic to $E_k$.  By Lemma \ref{unique}, the $L_k$ are in different $\Aut(A_m)$-equivalence classes.       

Finally, we need to show that any smooth and merely ample line bundle $M$ of degree $d$ is $\Aut(A_m)$-equivalent to one of the $L_k$.  By Reider's theorem, $M.F = 2$ for some elliptic curve which is isomorphic to $E_k$ for some $k | m$.  As $q_M^k = [ak', -2bm, ck]$ represents 2, we must have $(k,k') \leq 2$.  Since $m$ is odd, this means $(k,k') = 1$ and again we may assume $k = 1$.  Writing $L$ for $L_1$, we need to show that $L$ and $M$ are $\Aut(A_m)$-equivalent.  But both $q_L$ and $q_M$ are matrix-primitive and represent 2 and so are $\GL_2(\Z)$-equivalent to the form $[2,2,(1+md)/2]$.   In other words, $\Psi_{m,d}(L) =\Psi_{m,d}(M)$.  Since $\Psi_{m,d}$ is injective on the set of equivalence classes of $L$ for which $q_L$ is matrix-primitive, we see that $L$ and $M$ are $\Aut(A_m)$-equivalent.                 
\end{proof}

\begin{proposition}\label{mevennot8}
Let $d \geq 5$ be odd and squarefree and suppose $m$ is even but not divisible by $8$.  If $L \in \NS(A_m)$ has degree $d$ and is merely ample, then $L$ is not smooth.   
\end{proposition}

\begin{proof}
First assume that $m \equiv 2$ (mod 4).  As $L$ is merely ample, we have $L.F = 1$ or $L.F = 2$ for some elliptic curve $F \subset A_m$.  We may assume that $L.F = 2$ and we need to show that $L.F' = 1$ for some elliptic curve $F'$.  $F$ is isomorphic to $E_k$ for some $k | m$ and $q_L^k = [ak', -2bm,ck]$ represents 2.  This implies $(k,k') = 1$.  Thus after reindexing the elliptic curves on $A_m$, we may assume $k = 1$ and $q_L$ represents 2.  The content of $q_L$ is then at most 2, and since $\Disc(q_L) = -4md$ is not divisible by 16, $q_L$ is in fact primitive.  We must then have $[q_L] = [f_2],$ as there is only one class of quadratic form of discriminant $-4md$ which represents 2.  By Corollary \ref{equiv}, we have 
$$1 = [f_2]^2 = [q_L] \cdot [f_2] = [q_{w_2(L)}].$$
So there exists an elliptic curve $F_0 \subset A_m$ such that $w_2(L).F_0 = 1$.  Hence $L.w_2(F_0) = 1$, showing that $L$ is not smooth.  

Now assume that $m \equiv 4$ (mod 8) and again suppose $L.F = 2$ for an elliptic curve $F$ isomorphic to $E_k$.  As $q_L^k = [ak', -2bm, ck]$ represents 2 and $a$ and $c$ are odd (since $ac = mb^2 + d$), we must have $(k,k') = 2$ and $q_L^k$ has content 2.  After relabeling the elliptic curves, we may assume $k = 2$ and that $\frac{1}{2}q_L^2$ represents 1.  Write $q_L^2 = h(x,y)$.  Then
$$h(2x,y) = [am, -2bm, c] = q_L$$
$$h(x,2y) = [am/4, -2bm, 4c] = q^4_L$$
By Lemma \ref{picardmap}, $[q_L]$ and $[q_L^4]$ are in the kernel of 
$$e_2: \Pic(\O_{-4md}) \to \Pic(\O_{-md}).$$
This kernel has size 2 by Lemma \ref{kernel}.  But $[q_L] \neq [q_L^4]$ by Proposition \ref{equivariance}.  So one of $[q_L]$ or $[q_L^4]$ is the trivial class; i.e.\ one of $q_L$ or $q_L^4$ represents 1.  In particular, there is an elliptic curve $F' \subset A_m$ such that $L.F' = 1$, as desired.      
\end{proof}

\begin{proposition}\label{divby8}
Let $d \geq 5$ be odd and squarefree and suppose $m$ is divisible by $8$.  Then there are $2^{\omega(m)}$ $\Aut(A_m)$-equivalence classes of smooth and merely ample line bundles of degree $d$ on $A_m$.
\end{proposition}

\begin{proof}
First we construct, for each $k | m$ such that $(k,k') = 2$, a smooth and merely ample line bundle $L_k$ such that $L_k.F = 2$ for some elliptic curve $F$ isomorphic to $E_k$.  Note that there are $2^{\omega(m)}$ such divisors $k$ and that the $L_k$ are not $\Aut(A_m)$-equivalent to one another by Lemma \ref{unique}.   It suffices to construct $L_k$ when $k = 2$, because we can always relabel the elliptic curves on $A_m$.  

To construct $L_2$, we consider the following elliptic curves on $A_m = E \times E'$:
$$ E_2 \equiv 2h + X_\l + (m/2)v$$
$$F \equiv (m/2)h + (m/4 + 1) X_\l + 2\left( m/4 + 1\right)^2 v.$$
$E_2$ is the image of the map $E \to A_m$ given by $P \mapsto (2P, \l(P))$, and $F$ is the image of the map $P \mapsto (\frac{m}{2} P, \left(\frac{m}{4} + 1\right)\l(P)).$
These two elliptic curves intersect in $(E_2. F) = 4$ points.  In fact, the four points are given by $(0,0)$, $(0, \l(R))$, $(S,0)$ and $(S, \l(R))$, where $S$ generates the order 2 subgroup $H_2$ of $\ker \l$ and $S$ is any other order 2 point on $E$.  We therefore have $E_2[2] = F[2]$.  Now consider the subtraction map
$$\mu: B := E_2 \times F \to E \times E' = A_m$$
$$(P,Q) \mapsto P - Q.$$
By Theorem \ref{precisereider}, there is a merely ample $L_2 \in \NS(A_m)$ of degree $d$ such that $\mu^*L = 2dh_B + 2v_B$.  

On the other hand, $L_2$ is smooth.  For if otherwise, then $L_2 = \tilde h + d\tilde v$ with respect to some product decomposition of $A_m$, and as $(\tilde h + d\tilde v).E_2 = 2$, we must have $\tilde v.E_2 = 0$.  This forces $\tilde v = E_2$, but $E_2$ is not a direct factor of $A_m$ as $(2,m/2) > 1$, so we have a contradiction. 

It remains to show that any smooth and merely ample $M$ on $A_m$ of degree $d$ is $\Aut(A_m)$-equivalent to one of the $L_k$.  We have $M.F = 2$ for some elliptic curve $F$ isomorphic to $E_k$ for some $k | m$.  As $q_M^k$ represents 2, we must have $(k,k') = 2$ and we may, as usual, assume $k = 2$.  We now set $L = L_2$ and show that $M$ and $L$ are $\Aut(A_m)$-equivalent.  If we write $q_M^2 = 2h(x,y)$ for a quadratic form $h$ which represents 1, then $h(2x,y) = q_M$.  In particular, $q_M$ is primitive.  By Lemma \ref{picardmap}, $[q_M]$ is in the kernel of 
$$e_2 : \Pic(\O_{-4md}) \to \Pic(\O_{-md}),$$
which has size 2.  But $[q_M] = \Psi_{m,d}(M)$ cannot be trivial because the trivial class corresponds to the non-smooth polarization $dh + v$.  So $[q_M]$ is the generator of $\ker(e_2)$.  The same argument applies to $L_2$, so $\Psi_{m,d}(L_2) = \Psi_{m,d}(M)$.  Since $q_M$ and $q_{L_2}$ are primitive, this implies that $L_2$ and $M$ are $\Aut(A_m)$-equivalent.  
\end{proof}

Propositions \ref{oddm}, \ref{mevennot8}, and \ref{divby8} and their proofs suggest that if $L$ is merely ample, then $\Psi_{m,d}(L)$ is 2-torsion in its corresponding class group.  We prove this next in certain cases.  The proof is essentially the translation of the proofs above from the language of algebraic geometry to the language of quadratic forms via the correspondence of Theorem \ref{bijnoncm}.

\begin{proposition}\label{merely}
Suppose $d = p$ or $d = 2p$ for an odd prime $p$ and let $L \in \NS(A_m)$ be merely ample of degree $d \geq 5$.  Assume that $m$ is odd if $d$ is even.  Then $\Psi_{m,d}(L)$ is $2$-torsion in its class group.  
\end{proposition}   

\begin{proof}   
Write $L = ah + bX_\l + cv$.  If $L$ is not smooth, then $\Psi_{m,d}(L)$ is 2-torsion by Proposition \ref{2t}.  So assume $L$ is smooth and merely ample.  Then by Proposition \ref{crit} and Theorem \ref{reider}, there is an elliptic curve $E_k$ on $A_m$ such that $E_k. L = 2$.  Here, $E_k \cong E/H_k$ for some divisor $k$ of $m$ and $k' = k/m$.  This means that the quadratic form
$$q_L^k(x,y) = ak'x^2 - 2bmxy + cky^2$$
represents 2.

\begin{enumerate}
\item Case: $q_L^k$ is a multiple of 2.  
\begin{enumerate}
\item Subcase: $m$ is odd.  It follows that $a$ and $c$ are even, $b$ and $d$ are odd, and $(k,k') = 1$.  As $\frac{1}{2}q_L^k$ represents 1, its class in $\Cl(-md)$ is the trivial class.  By Lemma \ref{trans}, $[\frac{1}{2}q_{w_k(L)}]$ is also the trivial class.  By Corollary \ref{equiv} applied to $w_k(L)$, we have 
$$\Psi_{m,d}(L) = e_g(\Psi_{m,d}(w_k(L))) \cdot [g_k] = e_g(1) \cdot [g_k] = [g_k],$$
where $g = (k,d)$.  The class $[g_k]$ is 2-torsion by Lemma \ref{2torslemma}, which proves the proposition in this subcase.   

\item Subcase: $m$ is even and $(m,d) = 1$.  In particular, $d = ac - mb^2$ is prime.  In this case, both $a$ and $c$ are odd, so $k$ and $k'$ are even and $(k/2, k'/2) = 1$.  To ease notation, we write $D = -md$ and as usual write $\O_D$ for the quadratic order of discriminant $D$.  Consider the following quadratic forms of discriminant $4D$:   
$$f_1 = \left[k/2, \,0,\,  2k'd\right]$$
$$f_2 = [2k,\,0 \,,\,k'd/2].$$
At least one of these is primitive, call it $f_k$, and so it makes sense to consider the product $[f_k] \cdot [q_L]$ in the class group $\Cl(4D)$.  Dirichlet composition of forms gives $[f_k]\cdot [q_L] = [h_k]$, where $h_k$ is one of 
$$h_1 = \left[2ak',\, -2bm, \,kc/2\right],$$
$$h_2 = \left[ak'/2, \, -2bm, \, 2kc\right]$$

depending on whether we chose $f_1$ or $f_2$.  We claim that $[h_k]$ is 2-torsion in  $\Pic(\O_{4D})$.  In fact, $[h_k]$ is in the kernel of the natural map 
$$e_2: \Pic(\O_{4D}) \to \Pic(\O_{D}),$$ which establishes the claim as $\ker(e_2)$ has size 2 by Lemma \ref{kernel}.  To see that $[h_k] \in \ker(e_2)$, note that 
$$h_1(x,y) = \frac{1}{2}q_L^k(2x,y) \mbox{  and  } h_2(x,y) = \frac{1}{2}q_L^k(x,2y),$$
and use Lemma \ref{picardmap} and the fact that $\left[\frac{1}{2}q_L^k\right]$ is trivial in $\Pic(\O_D)$.   
Finally, $f_k$ is 2-torsion in $\Cl(4D)$ as well, so we see that $[q_L] = [f_k]\cdot [h_k]$ is 2-torsion in $\Cl(4D)$.       

\item Subcase: $m$ is even and $d$ is a prime $p$ which divides $m$.  If $(c,p) = 1$, then $q_L$ is primitive and we can argue as in the previous case.  So suppose $p$ divides $c$, hence $p$ divides $k$ as well.  The quadratic form $q_L$ has content $p$ and we let $q$ be the primitive form $\frac{1}{p}q_L$ of discriminent $-4m/p$.  As before, we have 
\begin{equation}\label{eq1}[k/2p, \, 0\,, 2k'] \cdot [q] = [h_k],\end{equation}
in the class group $\Cl(-4m/p)$, where $h_k$ is either 
$$\left[2ak'p^2, \, -2bm/p, \, kc/2 \right] \mbox{  or  } \left[ak'/2p^2, \, -2bm/p, \, 2kc \right],$$ whichever one is primitive.  Assume for simplicity that the former is primitive (a similar argument holds if only the other is primitive).  If we set $z(x,y) = \frac{1}{2}q_L^k(x,2y)$, then $h_k(x,y) = z(x/p, y)$.  We have maps
$$e_2 : \Pic(\O_{-4mp}) \to \Pic(\O_{-mp})$$ 
$$e_p : \Pic(\O_{-4mp}) \to \Pic(\O_{-4m/p}),$$
and one checks as before that $e_2([z]) = \left[\frac{1}{2}q_L^k\right] = 1$ and $e_p([z]) = [h_k]$.  The first equation implies that $[z]$ is two-torsion and thus $[h_k]$ is 2-torsion by the second equation. Finally, we deduce from (\ref{eq1}) that $[q] = [\frac{1}{p}q_L] = \Psi_{m,d}(L)$ is 2-torsion in $\Cl(-4m/p)$.   
\end{enumerate}

\item Case: $q_L^k$ is primitive:  In this case $(k,k') = 1$.  If $md$ is a multiple of 4, then one quickly gets a contradiction to the fact that $q_L^k$ is primitive and represents 2.  So there are two subcases to consider:
\begin{enumerate}
\item Subcase: $md$ is even.  In this case, $q_L^k$ corresponds to the class of an ideal $\a$ in $\O_{-4md}$ of norm 2, by Lemma \ref{reps}.  But there is a single prime above the rational prime 2 in the maximal order containing $\O_D$ and as $\a$ must be a prime ideal, $\a$ is uniquely determined.  We conclude that $q_L^k$ is $\GL_2(\Z)$-equivalent to the 2-torsion class $[2,0,md/2]$.  Corollary \ref{equiv} and Lemma \ref{trans} then give 
$$e_g\left([2,0,md/2]\right) \cdot [f_k] = \Psi_{m,d}(L),$$
which shows that $\Psi_{m,d}(L)$ is 2-torsion.   
\item Subcase: $md$ is odd.  First note that $md \equiv 1$ (mod 4) in this case, for if $md \equiv 3$ (mod 4), then there are no primitive quadratic forms of discriminant $-4md$ which represent 2.  
Since $md \equiv 1$ (mod 4), the prime 2 ramifies in the corresponding quadratic field.  We conclude as before that $[q_L^k]$ is the 2-torsion class $[2,2, (1+md)/2]$ and 
$$e_g\left([2,2,(1+ md)/2]\right) \cdot [f_k] = \Psi_{m,d}(L),$$
showing $\Psi_{m,d}(L)$ is 2-torsion.  
\end{enumerate}  
\end{enumerate}  
\end{proof}

\begin{notation}
We write $H(D)$ for the number of isomorphism classes of primitive integral symmetric bilinear forms of rank 2 and determinant $D$.   
\end{notation}

Of course, primitive integral symmetric bilinear forms of determinant $D$ are in bijection with matrix-primitive quadratic forms of discriminant $-4D$.  If one wants a bijection with \textit{primitive} integral quadratic forms, then primitive quadratic forms of discriminant $d \equiv 1$ (mod 4) correspond to bilinear forms of discriminant $-D$.  If $D> 0$, then we can relate $H(D)$ to the classical class numbers $h(d)$ of positive definite quadratic forms of discriminant $d$ up to $\SL_2(\Z)$-equivalence as follows: 
$$H(D) = \frac{1}{2}\begin{cases}
h(-4D) + h_2(-4D) &  D \not \equiv 3 \mbox{ (mod 4)}\\
h(-4D) + h_2(-4D) + h(-D) + h_2(-D) & D \equiv 3 \mbox{ (mod 4)}
\end{cases}.$$   
Here, $h_2(d)$ is the number of 2-torsion classes.  

\begin{notation}
We write $H_2(D)$ for the number of primitive integral symmetric bilinear forms of rank 2 and determinant $D$ which correspond to $2$-torsion classes of quadratic forms.   
\end{notation}

\begin{theorem}\label{vacount}
Suppose $p \geq 5$ is prime and $(m,p) = 1$.  Then $N_\va(A_m,p) = H(mp) - H_2(mp)$.  
\end{theorem}

\begin{proof}
This follows from a simple computation using Propositions \ref{oddm}, \ref{mevennot8}, \ref{divby8}, Corollary \ref{nonsm}, and Proposition \ref{2torslemma}.
\end{proof}

\begin{proof}[Proof of Theorem $\ref{main}$]
If $d$ is not prime or twice a prime, then we can write $d = pq$ with $p$ and $q$ both larger than 2.  Then $L = ph + qv$ is very ample, being the product of pull backs of very ample line bundles on $E$ and $E'$.  So $N_\va(A_m,d) > 0$ if $d$ is not a prime or twice an odd prime.  Theorem \ref{main} now follows from Theorems \ref{mainp} and \ref{main2p}, which we prove next.
\end{proof} 

\begin{theorem}\label{mainp}
If $d= p \geq 5$ is prime, then $N_\va(A_m,d) = 0$ if and only if $(m,d) = 1$ and $\Cl(-4md)$ is $2$-torsion.
\end{theorem}  

\begin{theorem}\label{main2p}\normalfont
If $p \geq 3$ is prime, then $N_\va(A_m,2p) = 0$ if and only if
\begin{enumerate}
\item $(m,p) = 1$.
\item $\Cl(-4md)$ is $2$-torsion.  
\item If we factor $m = \prod_p p^{a_p}$, then $a_2$ is either $0$, $2$, or $3$.  
\end{enumerate}
\end{theorem}

\begin{proof}[Proof of Theorem $\ref{mainp}$]
If $(m,p) = 1$, the theorem follows from Theorem \ref{vacount}.  Note that this is true even when $mp \equiv 3$ (mod 4) because if $\Cl(-4mp)$ is 2-torsion, then $\Cl(-mp)$, being a quotient of $\Cl(-4mp)$, is 2-torsion as well.  

So assume now that $d = p$ divides $m$. We will show that there is a very ample line bundle of degree $p$ on $A_m$.  We proceed case-by-case and use the letters $d$ and $p$ interchangeably.    
\begin{enumerate}
\item Case: $m$ is odd. 
\begin{enumerate}
\item Subcase: $md \equiv 1$ (mod 4).  In this case $\Psi_{m,d}$ maps to $\Cl(-4mp) \coprod \Cl(-4m/p)$.  It follows from the analysis in Case (2b) of Proposition \ref{merely} that any smooth but merely ample $L$ which maps to $\Cl(-4m/p)$ does not map to a class of the form $[a,0,b]$.  So only non-smooth $L$ map to such classes.  In fact, exactly half of the $2^{\omega(m)}$ non-smooth $L$ map to $\Cl(-4m/p)$: those corresponding to divisors $k$ divisible by $p$.  We have $h_2(-4m/p)= 2^{\omega(m/p)}$, so these non-smooth $L$ either map 2-to-1 or 1-to-1 to over the points of the form $[a,0,b] \in \Cl(-4m/p)$, depending on whether $p^2 | m$ or not.  On the other hand, the fibers of the surjective map $\Psi_{m,d}$ are of size $\#\Gamma_0(m)\backslash \Gamma_0(m/p)/\Aut(q)$.  The size of $\Gamma_0(m)\backslash\Gamma_0(m/p)$ is either $p$ or $p+1$, depending on whether $p^2 | m$ or not.  As $q$ has even discriminant, $\Aut(q)$ has either 4 or 8 elements, but it acts on the coset space through a quotient of size at most 2.\footnote{If $|\Aut(q)| = 8$, then $[q] = [x^2 + y^2]$ and there are 4 diagonal automorphisms.}  Since $p \geq 5$, we see that in all cases, the fibers of $\Psi_{m,d}$ above points of the form $[a,0,b] \in \Cl(-4m/p)$ are larger than 2, so there must be a very ample line bundle in each of those fibers.  So $N_\va(A_m,d) > 0$ in this subcase, as desired.   
\item Subcase: $md \equiv 3$ (mod 4).  This time $\Psi_{m,d}$ maps onto 
$$\Cl(-4mp) \coprod \Cl(-mp) \coprod \Cl(-4m/p)\coprod \Cl(-m/p).$$
However the $2^{\omega(m)}$ non-smooth $L$ map to the even discriminant groups and the rest of the merely ample $L$ map to the odd discriminant groups, by the analysis in Case (1a) of Proposition \ref{merely}.  But 
$$h_2(-4mp) + h_2(-4m/p) =  2^{\omega(m)} +  h_2(-4m/p) > 2^{\omega(m)},$$ so there must be at least one very ample $L$ of degree $d$, i.e. $N_\va(A_m,d)  > 0$.  
\end{enumerate}

\item Case: $m$ even, not divisible by 8.  In this case, $\Psi_{m,d}$ maps onto 
$$\Cl(-4mp) \coprod \Cl(-4m/p).$$  We have $h_2(-4mp) = 2^{\omega(m) -1}$ and $h_2(-4mp)$ is either $2^{\omega(m) - 1}$ or $2^{\omega(m) - 2}$, depending on whether $p^2 | m$ or not.  But if $L\in \NS(A_m)$ has degree $p$ and is merely ample, then it is also not smooth, by Proposition \ref{mevennot8}.  So on the one hand, the merely ample $L$ map either 1-to-1 or 2-to-1 onto $\Cl(-4m/p)$, whereas the fibers of $\Psi_{m,d}$ above such points have size $\#\Gamma_0(m)\backslash \Gamma_0(m/p)/\Aut(q) > 2$.  There therefore must be some very ample $L$ of degree $p$ and $N_\va(A_m,d) > 0$.  

\item  Case: $m \equiv 0$ (mod 8).  In this case, $\Psi_{m,d}$ maps onto
$$\Cl(-4mp) \coprod \Cl(-4m/p).$$  By Proposition \ref{divby8} and Corollary \ref{nonsm}, there are exactly $2^{\omega(m) + 1}$ merely ample classes of line bundles.  On the other hand, $\#\Cl(-4mp)[2] = 2^{\omega(m)}$ and $\#\Cl(-4m/p)[2]$ is either $2^{\omega(m)}$ or $2^{\omega(m) -1 }$, depending on whether $p^2 | m$ or not.  The fibers of $\Psi_{m,d}$ above $\Cl(-4mp)[2]$ are singletons and the fibers above $q \in \Cl(-4m/p)[2]$ are of size 
$$\Gamma_0(m) \backslash \Gamma_0(m/p) \Aut(q).$$
The coset space $\Gamma_0(m) \backslash \Gamma_0(m/p)$ has size $p$ or $p+ 1$ depending on whether $p^2 | m$, and $\Aut(q)$ acts through a quotient of size at most 2.  Thus the fibers of $\Psi_{m,d}$ above $\Cl(-4m/p)$ have size at least $p/2 > 2$.  This shows that there are more than $2^{\omega(m) + 1}$ degree $p$ polarizations on $A_m$, hence $N_\va(A_m,p) > 0$, as desired.  
\end{enumerate}
\end{proof}

\begin{proof}[Proof of Theorem $\ref{main2p}$] 
We again break into several cases.  
\begin{enumerate}
\item Case: $m$ odd and $(p,m) =1$. In this case, $\Psi_{m,d}$ maps the isomorphism classes of degree $d$ polarizations $L \in \NS(A_m)$ bijectively onto $\Cl(-4md)$.  The merely ample $L$ map to $\Cl(-4md)[2]$, which has size $2^{\omega(m) + 1}$.  The non-smooth $L$ map to the $2^{\omega(m)}$ forms $[k,0,k'd]$.  The smooth merely ample line bundles must map to the remaining 2-torsion classes by Proposition \ref{merely}; these are of the form $[2k,0,k'p]$.  Conversely, if $[q_L] = [2k,0,k'p]$ for $k$ such that $(k,k') = 1$, then $L$ is smooth and merely ample.  Indeed, by Proposition \ref{equivariance}, $[q^k_L] = [q_L]\cdot [f_k] = [2,0,mp]$, which represents 2.  So $L.E_k = 2$ and $L$ isn't very ample.  Altogether, we have accounted for the $2^{\omega(m) + 1}$ 2-torsion classes in $\Cl(-4md)$, showing that 
$$N_\va(A_m,d) = \frac{1}{2}\left[h(-4md) - h_2(4md)\right].$$   

\item Case: $m$ odd and $p$ divides $m$.  
Let $f: A_m \to A_{2m}$ be any degree 2 isogeny.  For example, let $A_{2m} = E \times E''$, where $E'' = E'/\langle P \rangle$ with $P \in E'$ any point of order 2 and set 
$$f = \id \times \pi: A_m = E \times E' \to E \times E'' = A_{2m}.$$  By Theorem \ref{mainp}, there exists a very ample line bundle $M$ of degree $p$ on $A_{2m}$.  Then $f^*M \in \NS(A_m)$ has degree $2p$ and is very ample by Lemma \ref{vapullback}.  When $p = 3$, there is no very ample $M$, for degree reasons.  However, it is still true that there exists an $M$ for which $M.F > 2$ for all elliptic curves $F \subset A_{2m}$.  This is enough for the proof of Lemma \ref{vapullback} to go through.\footnote{Here and in some other cases, we leave the details for the case $p = 3$ to the reader.}  So $N_\va(A_m, d) > 0$.              

\item Case: $m \equiv 2$ (mod 4).  If $p | m$, then $N_\va(A_m,d) > 0$.  Indeed, we may choose an isogeny $f: A_m \to A_{2m}$ and by Theorem \ref{mainp} , there is a very ample line bundle $M \in \NS(A_{m/2})$ of degree $p$.  Then $f^*M \in \NS(A_m)$ is very ample of degree $2p = d$.     

Assume now that $(p,m) = 1$.   We will construct a very ample $L$ on $A_m$ of degree $2p$ using the 2-isogeny 
$$f : A_m = E \times E' \stackrel{\id \times \hat \mu_{m/2}}\longrightarrow E \times E_{m/2} = A_{m/2}.$$  Let $M \in \NS(A_{m/2})$ be smooth of degree $p$ and satisfying $M.F = 2$ for some elliptic curve $F \subset A_{m/2}$ which is abstractly isomorphic to $E_{m/2}$.  The existence of such an $M$ follows from the proof of Proposition \ref{oddm}.  By Proposition \ref{ECs}, we may choose $\alpha  \in \Aut(A_{m/2})$ such that $\alpha^*F = \{0\} \times E_{m/2}$.  Set $g = \alpha \circ f$, so that $g: A_m \to A_{m/2}$ is a 2-isogeny.  We have $g^*F = v$ and $g_*v = 2F$.  Thus,
$$(g^*M.v) = (M.g_*v) = 2(M.F) = 4.$$
On the other hand, $g^*M$ is smooth because $M$ is smooth.  So if $g^*M$ is not very ample, then there exists an elliptic curve $F' \subset A$ such that $(g^*M.F') = 2$ and so $(M.g_*F') = 2$.  Note that $g_*F'$ is an elliptic curve for otherwise it is twice an elliptic curve and this would contradict the smoothness of $M$.   By Lemma \ref{unique}, we must have $g_*F' = F$.  This shows that $F' = v$ and hence $(g^*M.v) = 2$, which contradicts our computation above.  So $g^*M$ is very ample and $N_\va(A_m,d) > 0$.       

\item Case: $m \equiv 0$ (mod 4).
\begin{enumerate}
\item Subcase: $p$ divides $m$.  Consider a 2-isogeny $f: A_m \to A_{m/2}$.  By Theorem \ref{mainp}, there is a very ample $L \in \NS(A_{m/2})$ of degree $p$, so $f^*L$ is very ample of degree $2p$.  Hence $N_\va(A_m,d) > 0$.  
\item Subcase: $(p,m) = 1$ and $m \not \equiv 0$ (mod 16).  Then $\Psi_{m,d}$ maps onto 
$$\Cl(-8mp) \coprod \Cl(-2mp).$$
If $L \equiv ah + bX_\l + cv$ maps to $\Cl(-8mp)$, then $a$ is even and $L = f^*M$, where $f$ is the 2-isogeny 
$$A_m = E \times E' \to A_{m/2} = E \times E_{m/2}$$
and $M$ has degree $p$.  If $L$ is not very ample, then $M$ is not very ample and by Proposition \ref{mevennot8}, we see that $M$ is not even smooth.  Conversely, if $M \in \NS(A_{m/2})$ of degree $p$ is not smooth, then $M.F = 1$ for some elliptic curve $F \subset A_{m/2}$ and therefore $f^*M.f^*F = 2$.  So $f^*M$ has degree $2p$ and is not very ample.  There are $2^{\omega(m)}$ such $M$ up to $\Aut(A_{m/2})$-equivalence, and they give rise to $2^{\omega(m) + 1}$ $\Aut(A_m)$-equivalence classes of non-very-ample $L$ of degree $2p$, because $[\Gamma_0(m/2) : \Gamma_0(m)] = 2$ and $\Aut(q_M)$ acts trivially on the cosets.  

Similarly, if $\Psi_{m,d}(q_L) \in \Cl(-2mp)$, then $c$ is even and $L = g^*M$ where $g$ is the 2-isogeny
$$A_m  = E \times E' \to E_2\times E' = A_{m/2},$$
and $M$ has degree $p$.  There are again $2^{\omega(m) + 1}$ different $L$ which are not very ample.  On the other hand, $h_2(-8mp) = 2^{\omega(m) + 1}$ and $h_2(-2mp) = 2^{\omega(m)}$ with fibers of $\Psi_{m,d}$ above points in $\Cl(-2mp)$ having size exactly 2.  So $N_\va(A_m,d) = 0$ if and only if both $\Cl(-4md)$ and $\Cl(-md)$ are 2-torsion if and only if $\Cl(-4md)$ is 2-torsion.
\item Subcase: $(p,m) = 1$ and $m \equiv 0$ (mod 16).  We will construct a very ample divisor of degree $2p$ on $A_m$.  We consider the $2$-isogeny
$$f: A_m  = E \times E' \xrightarrow{\l_2 \times \id} E_2 \times E' = A_{m/2}.$$
We write $\l_0: E_2 \to E'$ for the unique map satisfying $\l = \l_0 \circ \l_2$, i.e. $\l_0$ is a minimal isogeny connecting $E_2$ and $E'$.  By Proposition \ref{divby8}, $A_{m/2}$ has a smooth line bundle $L$ of degree $p$ such that $L.E_4 = 2$, where $E_4 \subset A_{m/2}$ is the image of the map 
$$E_2 \to E_2 \times E'$$
$$R \mapsto (2R, \l_0(R)).$$
We refer to it as $E_4$ because it is abstractly isomorphic to the elliptic curve we've been calling $E_4$ on $A_{m}$.  With respect to the surface $A_{m/2}$, though, this is the elliptic curve we called $E_2$ in the proof of Proposition \ref{divby8}.  We think of $E_2 \subset A_m$ as the image of the map $E \to A_m : P \mapsto (2P, \l(P))$ as usual.  It is easy to check that $f(E_2) = E_4$.  As $E_2$ and $E_4$ are not abstractly isomorphic, this forces the induced map $f: E_2 \to E_4$ to have degree 2.  So $f^*E_4 = E_2$ and $f_*E_2 = 2E_4$, giving 
$$(f^*L.E_2) = (L.2E_4) = 4.$$
But if $F \subset A_m$ is an elliptic curve on $A_m$ such that $(f^*L.F) = 2$, then $(L.f_*F) = 2$ and $f_*F = E_4$ by Lemma \ref{unique}.  But this forces $F = E_2$, and hence a contradiction:
$$4 = (f^*L.E_2) = (f^*L.F) = 2.$$  We conclude that $f^*L$ is very ample of degree $2p$ and $N_\va(A_m,d) > 0$.      
\end{enumerate}
\end{enumerate}
\end{proof}
%
%
%

\end{document}